\documentclass[12pt]{article}

\usepackage{geometry}
\usepackage{hyperref} 

\usepackage{amsmath}
\usepackage{amssymb}
\usepackage{amscd}
\usepackage{textcomp}
\usepackage{dsfont}
\usepackage{mathrsfs}

\long\def\forget#1{}

\newcommand{\Verkuerzung}[2]{#1}

\newcommand{\lang}[1]{\mbox{#1}}

\usepackage{color}
\newcounter{commentcounter}
\newcommand{\comment}[1]{\stepcounter{commentcounter}{\color{red}\textbf{Comment \arabic{commentcounter}.} #1}
\immediate\write16{}
\immediate\write16{Warning: There was still a comment . . . }
\immediate\write16{}}

\newcounter{urscommentcounter}

\def\?{\ 
{\bf\color{red}???}\ 
\immediate\write16{}
\immediate\write16{Warning: There was still a question mark . . . }
\immediate\write16{}}

\usepackage{amsthm}
\theoremstyle{plain}
\newtheorem{theorem}{Theorem}[section]

\newtheorem{corollary}[theorem]{Corollary}
\newtheorem{proposition}[theorem]{Proposition}

\theoremstyle{definition}
\newtheorem{definition}[theorem]{Definition}
\newtheorem{definition-theorem}[theorem]{Definition-Theorem}
\newtheorem{definition-remark}[theorem]{Definition-Remark}

\newtheorem{remark}[theorem]{Remark}

\theoremstyle{remark}

\usepackage{xy}
\xyoption{all}

\newdir^{ (}{{}*!/-3pt/\dir^{(}}    
\newdir^{  }{{}*!/-3pt/\dir^{}}    
\newdir_{ (}{{}*!/-3pt/\dir_{(}}    
\newdir_{  }{{}*!/-3pt/\dir_{}}

\newcounter{zahl}

\def\theenumi{(\alph{enumi})}

\def\p@enumii{\theenumi}

\newcommand{\DS}{\displaystyle}
\newcommand{\TS}{\textstyle}
\newcommand{\SC}{\scriptstyle}
\newcommand{\SSC}{\scriptscriptstyle}

\newcommand{\cM}{\mathcal{M}}

\DeclareMathOperator{\Aut}{Aut}

\DeclareMathOperator{\Gal}{Gal}
\DeclareMathOperator{\GL}{GL}

\DeclareMathOperator{\Koh}{H}

\DeclareMathOperator{\Hom}{Hom}

\DeclareMathOperator{\Isom}{Isom}

\DeclareMathOperator{\Quot}{Frac}

\DeclareMathOperator{\Rep}{Rep}

\DeclareMathOperator{\SL}{SL}

\DeclareMathOperator{\Spec}{Spec}
\DeclareMathOperator{\Spf}{Spf}

\DeclareMathOperator{\Tr}{Tr}

\newcommand{\alg}{{\rm alg}}

\newcommand{\fppf}{{\it fppf\/}}
\newcommand{\fpqc}{{\it fpqc\/}}

\DeclareMathOperator{\id}{\,id}

\newcommand{\red}{{\rm red}}

\newcommand{\sep}{{\rm sep}}

\newcommand{\mot}{{\cM ot_C^{\ul \nu}}}

\newcommand{\scrA}{{\mathscr{A}}}

\newcommand{\scrH}{{\mathscr{H}}}

\newcommand{\scrS}{{\mathscr{S}}}

\DeclareMathOperator{\whtimes}{\mathchoice
            {\wh{\raisebox{0ex}[0ex]{$\DS\times$}}}
            {\wh{\raisebox{0ex}[0ex]{$\TS\times$}}}
            {\wh{\raisebox{0ex}[0ex]{$\SC\times$}}}
            {\wh{\raisebox{0ex}[0ex]{$\SSC\times$}}}}

\renewcommand{\phi}{\varphi}
\renewcommand{\epsilon}{\varepsilon}

\usepackage{amsfonts}
\newcommand{\BOne} {{\mathchoice{\hbox{\rm1\kern-2.7pt l\kern.9pt}}
                              {\hbox{\rm1\kern-2.7pt l\kern.9pt}}
                              {\hbox{\scriptsize\rm1\kern-2.3pt l\kern.4pt}}
                              {\hbox{\scriptsize\rm1\kern-2.4pt l\kern.5pt}}}}

\newcommand{\BA}{{\mathbb{A}}}

\newcommand{\BD}{{\mathbb{D}}}

\newcommand{\BF}{{\mathbb{F}}}
\newcommand{\BG}{{\mathbb{G}}}

\newcommand{\BL}{{\mathbb{L}}}

\newcommand{\BN}{{\mathbb{N}}}

\newcommand{\BP}{{\mathbb{P}}}
\newcommand{\BQ}{{\mathbb{Q}}}
\newcommand{\BR}{{\mathbb{R}}}
\newcommand{\BS}{{\mathbb{S}}}

\newcommand{\BZ}{{\mathbb{Z}}}

\newcommand{\CA}{{\cal{A}}}

\newcommand{\CF}{{\cal{F}}}
\newcommand{\CG}{{\cal{G}}}

\newcommand{\CI}{{\cal{I}}}

\newcommand{\CL}{{\cal{L}}}
\newcommand{\CM}{{\cal{M}}}
\newcommand{\CN}{{\cal{N}}}
\newcommand{\CO}{{\cal{O}}}
\newcommand{\CP}{{\cal{P}}}

\newcommand{\CS}{{\cal{S}}}
\newcommand{\CT}{{\cal{T}}}

\newcommand{\CV}{{\cal{V}}}
\newcommand{\CW}{{\cal{W}}}
\newcommand{\CX}{{\cal{X}}}

\newcommand{\CZ}{{\cal{Z}}}

\newcommand{\FG}{{\mathfrak{G}}}

\newcommand{\FM}{{\mathfrak{M}}}

\newcommand{\FP}{{\mathfrak{P}}}
\newcommand{\FQ}{{\mathfrak{Q}}}

\newcommand{\FX}{{\mathfrak{X}}}

\newcommand{\LM}{\textbf{M}^{loc}}

\let\setminus\smallsetminus

\newcommand{\es}{\enspace}

\newcommand{\dual}{^{\SSC\lor}}

\newcommand{\ul}[1]{{\underline{#1}}}
\newcommand{\ol}[1]{{\overline{#1}}}
\newcommand{\wh}[1]{{\widehat{#1}}}
\newcommand{\wt}[1]{{\widetilde{#1}}}

\newcommand{\SSS}{S}

\usepackage{ifthen}

\newcommand{\invlim}[1][]{\ifthenelse{\equal{#1}{}}
{\DS \lim_{\longleftarrow}}
{\DS \lim_{\underset{#1}{\longleftarrow}}}
}

\newcommand{\dirlim}[1][]{\ifthenelse{\equal{#1}{}}
{\DS \lim_{\longrightarrow}}
{\DS \lim_{\underset{#1}{\longrightarrow}}}
}

\newcommand{\dbl}{{\mathchoice{\mbox{\rm [\hspace{-0.15em}[}}
                              {\mbox{\rm [\hspace{-0.15em}[}}
                              {\mbox{\scriptsize\rm [\hspace{-0.15em}[}}
                              {\mbox{\tiny\rm [\hspace{-0.15em}[}}}}
\newcommand{\dbr}{{\mathchoice{\mbox{\rm ]\hspace{-0.15em}]}}
                              {\mbox{\rm ]\hspace{-0.15em}]}}
                              {\mbox{\scriptsize\rm ]\hspace{-0.15em}]}}
                              {\mbox{\tiny\rm ]\hspace{-0.15em}]}}}}
\newcommand{\dpl}{{\mathchoice{\mbox{\rm (\hspace{-0.15em}(}}
                              {\mbox{\rm (\hspace{-0.15em}(}}
                              {\mbox{\scriptsize\rm (\hspace{-0.15em}(}}
                              {\mbox{\tiny\rm (\hspace{-0.15em}(}}}}
\newcommand{\dpr}{{\mathchoice{\mbox{\rm )\hspace{-0.15em})}}
                              {\mbox{\rm )\hspace{-0.15em})}}
                              {\mbox{\scriptsize\rm )\hspace{-0.15em})}}
                              {\mbox{\tiny\rm )\hspace{-0.15em})}}}}

\newcommand{\dotBD}{\vbox{\hbox{\kern2pt\bf.}\vskip-4.5pt\hbox{$\BD$}}}

\DeclareMathOperator{\QIsog}{QIsog}
\DeclareMathOperator{\Nilp}{\CN \!{\it ilp}}

\DeclareMathOperator{\Sets}{\CS \!{\it ets}}

\def\s{\sigma^\ast}

\def\longto{\longrightarrow}
\def\isoto{\stackrel{}{\mbox{\hspace{1mm}\raisebox{+1.4mm}{$\SC\sim$}\hspace{-3.5mm}$\longrightarrow$}}}

\newbox\mybox
\def\arrover#1{\mathrel{
       \setbox\mybox=\hbox spread 1.4em{\hfil$\scriptstyle#1$\hfil}
       \vbox{\offinterlineskip\copy\mybox
             \hbox to\wd\mybox{\rightarrowfill}}}}

\newcommand{\BaseOfD}{\BF}

\newcommand{\BaseFldInSectUnif}{k}

\newcommand{\genericG}{P}

\DeclareMathOperator{\SpaceFl}{\CF\ell}
\newcommand{\tauGlob}{\tau}
\newcommand{\tauLoc}{\hat\tau}
\newcommand{\charsect}{s}

\begin{document}

\author{Esmail Arasteh Rad\forget{\footnote{Part of this research was carried out while I was visiting Institute For Research In Fundamental Sciences (IPM).}} }

\date{\today}

\title{Rapoport-Zink Spaces For Local $\BP$-Shtukas and Their Local Models\\}

\maketitle

\begin{abstract}
In this article we first survey the analogy between Shimura varieties (resp. Rapoport-Zink spaces) and moduli stacks for global $\FG$-shtukas (resp. Rapooprt Zink spaces for local $\BP$-shtukas).  This part is intended to enrich the dictionary between the arithmetic of number fields and function fields a bit further. Furthermore, to complete this picture, we also study some local properties of Rapoport-Zink spaces for local $\BP$-shtukas by constructing local models for them. This provides a ``local'' complementary to a previous  work of the author, which was devoted to the study of the local models for the moduli stacks of global $\FG$-shtukas. We also discuss some of its applications. 

\end{abstract}

\tableofcontents 

\section*{Introduction}

Recall that Shimura varieties come equipped with many symmetries. The significance of these symmetries is that they encode lots of arithmetic data. Furthermore, for wide range of cases, they appear as moduli spaces for motives, according to Deligne's conception of Shimura varieties \cite{Deligne1} and \cite{Deligne2}. From this perspective, it is expected that the Langlands correspondence will be realized on their cohomology.\\

There is an analogous picture over function fields. Let $\BP$ be a smooth affine group scheme of finite type over $\BD=\BF_q\dbl z\dbr$ with connected reductive generic fiber. A Rapoport-Zink space for local $\BP$-shtukas parametrizes (bounded) local $\BP$-shtukas together with a quasi-isogeny to a fixed local $\BP$-shtuka $\ul\BL$. To justify the naming, recall that local $\BP$-shtukas are function fields analogs for $p$-divisible groups. The Rapoport-Zink spaces are local counterparts of Shimura varieties, which in particular indicates that local Langlands correspondence may eventually be realized on their cohomology.\\ In contrast with number fields side, one may view Rapoport-Zink spaces for local $\BP$-shtukas (resp. moduli stacks of global $\FG$-shtukas) as function fields analogs for Rapoport-Zink spaces for $p$-divisible groups (resp. Shimura varieties).\\ 
The Rapoport-Zink spaces for local $\BP$-shtukas were first constructed and studied in \cite{H-V} for constant group $\BP$, i.e. $\BP=G\times_{\BF_q} \BD$ for a split reductive group $G$ over $\BF_q$, and then have been generalized to the case where $\BP$ is only a smooth affine group scheme over $\BD$ with connected reductive generic fiber in \cite{AH_Local}. \\

In this article we discuss some aspects of the above analogy, including uniformization theory and local model theory.  This part is intended to enrich the dictionary between the arithmetic of number fields and function fields a bit further.
In addition, to complete this picture, we also study local geometry of the Rapoport-Zink spaces for local $\BP$-shtukas. This is done by constructing local models for them; see Theorem \ref{ThmRapoportZinkLocalModel}. From this perspective, this article provides a short ``local'' complementary to \cite{AH_LM}, where the authors established the theory of local models for moduli of global $\FG$-shtukas, both in the sense of Beilinson-Drinfeld-Gaitsgory-Varshavsky, and also in the sense of Rapoport-Zink, in the following general setup. Namely, in \cite{AH_LM} the authors treat the case where $\FG$ is a smooth affine group scheme over a smooth projective curve $C$ over $\BF_q$. Here we also treat the general case, namely, we only assume that $\BP$ is a smooth affine group scheme over $\BD$. It can be easily seen that this assumption can not be weakened further. The local model theorem for Rapoport-Zink spaces for local $\BP$-shtukas has several immediate consequences. We discuss some of these applications in the last chapter \ref{SectApplocations}. Obviously, it clarifies type of singularities in certain cases. It also helps to answer the questions related to flatness. Recall that although the flatness of Rapoport-Zink spaces had been expected by Rapoport and Zink in \cite{RZ}, it was later observed that the flatness might fail in general. According to this observation, they introduced the notation $\textbf{M}^{naive}$ for the corresponding local model. To achieve the flatness, one should apply certain modifications which leads to the construction of the true local model $\LM$ inside $\textbf{M}^{naive}$. For the corresponding question for the moduli stacks of global $\FG$-shtukas see \cite[Chapter~4]{AH_LM}. Similarly, the local model for Rapoport-Zink spaces for local $\BP$-shtukas provides a criterion for flatness of these spaces over their reflex rings. Apart from local consequences of the local model theory, it has also some global consequences. This is because the local model diagram is defined globally. For example, it can be used to construct Kottwitz-Rapoport stratification on the special fiber of Rapoport-Zink spaces for local $\BP$-shtukas, which is useful for cohomology computations. Moreover, the local model theory can be implemented to study the semi-simple trace of Frobenius on the cohomology of these spaces, namely, by relating it to the better understood semi-simple trace of Frobenius on the cohomology of certain Schubert varieties (given by local boundedness conditions) inside twisted affine flag varieties,  see \cite{HR1}. The theorem can further be applied to produce certain cycles in the cohomology of the generic fiber of the Rapoport-Zink spaces for local $\BP$-shtukas, see section \ref{SectApplocations}.  \\

\section*{Acknowledgment}\label{Acknowledgment}

I warmly thank Urs Hartl for numerous stimulating conversations and insightful explanations, as well as his continues encouragements and support. Also I would like to thank Jakob Scholbach and Mohsen Asgharzadeh for inspiring conversations and valuable comments. 
I profited a lot from lectures given by Sophie Morel at IPM during Winter 2016 and Spring 2017 and wish to express my deep appreciation. This paper grew out of lectures delivered at a conference at IPM (April-May 2017) and for that I would like to thank the organizers and the staff. 

\section{Notation and Conventions}\label{SectNotation and Conventions}

Throughout this article we denote by
\begin{tabbing}

$\genericG_\nu:=\FG\times_C\Spec Q_\nu,$\; \=\kill

$\BF_q$\> a finite field with $q$ elements of characteristic $p$,\\[1mm]
$C$\> a smooth projective geometrically irreducible curve over $\BF_q$,\\[1mm]
$Q:=\BF_q(C)$\> the function field of $C$,\\[1mm]

$\BQ$\> the ring of rational numbers,\\[1mm]

$\BQ_p$\> the field of p-adic numbers for a prime $p\in \BZ$,\\[1mm]

$\BA:=\BA_\BQ$\> the ring of rational adeles associated with $\BQ$,\\[1mm]

$\BA_f$\> the ring of finite adeles,\\[1mm]

$\BA_f^p$\> the ring of finite adeles away from $p$,\\[1mm]

$\BaseOfD$\> a finite field containing $\BF_q$,\\[1mm]

$\wh A:=\BF\dbl z\dbr$\>  the ring of formal power series in $z$ with coefficients in $\BF$ ,\\[1mm]
$\wh Q:=\Quot(\wh A)$\> its fraction field,\\[1mm]

$\nu$\> a closed point of $C$, also called a \emph{place} of $C$,\\[1mm]
$\BF_\nu$\> the residue field at the place $\nu$ on $C$,\\[1mm]

$A_\nu$\> the completion of the stalk $\CO_{C,\nu}$ at $\nu$,\\[1mm]

$Q_\nu:=\Quot(A_\nu)$\> its fraction field,\\[1mm]

$\BD_R:=\Spec R\dbl z \dbr$ \> \parbox[t]{0.79\textwidth}{\Verkuerzung
{
the spectrum of the ring of formal power series in $z$ with coefficients in an $\BaseOfD$-algebra $R$,
}

{}
}

\Verkuerzung
{
\\[1mm]
$\hat{\BD}_R:=\Spf R\dbl z \dbr$ \> the formal spectrum of $R\dbl z\dbr$ with respect to the $z$-adic topology.
}
{}
\end{tabbing}

\noindent
For a formal scheme $\wh S$ we denote by $\Nilp_{\wh S}$ the category of schemes over $\wh S$ on which an ideal of definition of $\wh S$ is locally nilpotent. We  equip $\Nilp_{\wh S}$ with the \'etale topology. We also denote by
\begin{tabbing}
$\genericG_\nu:=\FG\times_C\Spec \wh Q_\nu,$\; \=\kill
$n\in\BN_{>0}$\> a positive integer,\\[1mm]
$\ul \nu:=(\nu_i)_{i=1\ldots n}$\> an $n$-tuple of closed points of $C$,\\[1mm]
$\BA_C^\ul\nu$\> the ring of rational adeles of $C$ outside $\ul\nu$,\\[1mm]
$A_\ul\nu$\> the completion of the local ring $\CO_{C^n,\ul\nu}$ of $C^n$ at the closed point $\ul\nu=(\nu_i)$,\\[1mm]

$\Nilp_{A_\ul\nu}:=\Nilp_{\Spf A_\ul\nu}$\> \parbox[t]{0.79\textwidth}{the category of schemes over $C^n$ on which the ideal defining the closed point $\ul\nu\in C^n$ is locally nilpotent,}\\[2mm]
$\Nilp_{\BaseOfD\dbl\zeta\dbr}$\lang{$:=\Nilp_{\hat\BD}$}\> \parbox[t]{0.79\textwidth}{the category of $\BD$-schemes $S$ for which the image of $z$ in $\CO_S$ is locally nilpotent. We denote the image of $z$ by $\zeta$ since we need to distinguish it from $z\in\CO_\BD$.}\\[2mm]
$\FG$\> a smooth affine group scheme of finite type over $C$,\\[1mm]
$\BP_\nu:=\FG\times_C\Spec  A_\nu,$ \> the base change of $\FG$ to $\Spec A_\nu$,\\[1mm]
$\genericG_\nu:=\FG\times_C\Spec Q_\nu,$ \> the generic fiber of $\BP_\nu$ over $\Spec Q_\nu$,\\[1mm]

$\BP$\> a smooth affine group scheme of finite type over $\BD=\Spec\BaseOfD\dbl z\dbr$,\\[1mm] 
$\genericG$\> the generic fiber of $\BP$ over $\Spec\BaseOfD\dpl z\dpr$.
\end{tabbing}

\noindent
Let $S$ be an $\BF_q$-scheme and consider an $n$-tuple $\ul s:=(s_i)_i\in C^n(S)$. We denote by $\Gamma_\ul s$ the union $\bigcup_i \Gamma_{s_i}$ of the graphs $\Gamma_{s_i}\subseteq C_S$. \\

\noindent
For an affine closed subscheme $Z$ of $C_S$ with sheaf $\CI_Z$ we denote by $\BD_S(Z)$ the scheme obtained by taking completion along $Z$\forget{ and by $\BD_{S,n}(Z)$ the closed subscheme of $\BD_S(Z)$ which is defined by $\CI_Z^n$}. Moreover we set $\dot{\BD}_S(Z):=\BD_S(Z)\times_{C_S} (C_S\setminus Z)$.\\

\noindent
We denote by $\sigma_S \colon  S \to S$ the $\BF_q$-Frobenius endomorphism which acts as the identity on the points of $S$ and as the $q$-power map on the structure sheaf. Likewise we let $\hat{\sigma}_S\colon S\to S$ be the $\BaseOfD$-Frobenius endomorphism of an $\BaseOfD$-scheme $S$. We set
\begin{tabbing}
$\genericG_\nu:=\FG\times_C\Spec Q_\nu,$\; \=\kill
$C_S := C \times_{\Spec\BF_q} S$ ,\> and \\[1mm]
$\sigma := \id_C \times \sigma_S$.
\end{tabbing}

 Assume that the generic fiber $\genericG$ of $\BP$ over $\Spec\BaseOfD\dpl z\dpr$ is connected reductive. Consider the base change $\genericG_L$ of $\genericG$ to $L=\BaseOfD^\alg\dpl z\dpr$. Let $S$ be a maximal split torus in $\genericG_L$ and let $T$ be its centralizer. Since $\BaseOfD^\alg$ is algebraically closed, $\genericG_L$ is quasi-split and so $T$ is a maximal torus in $\genericG_L$. Let $N = N(T)$ be the normalizer of $T$ and let $\CT^0$ be the identity component of the N\'eron model of $T$ over $\CO_L=\BaseOfD^\alg\dbl z\dbr$.

The \emph{Iwahori-Weyl group} associated with $S$ is the quotient group $\wt{W}= N(L)\slash\CT^0(\CO_L)$. It is an extension of the finite Weyl group $W_0 = N(L)/T(L)$ by the coinvariants $X_\ast(T)_I$ under $I=\Gal(L^\sep/L)$:
$$
0 \to X_\ast(T)_I \to \wt W \to W_0 \to 1.
$$
By \cite[Proposition~8]{H-R} there is a bijection
\begin{equation}\label{EqSchubertCell}
L^+\BP(\BaseOfD^\alg)\backslash L\genericG(\BaseOfD^\alg)/L^+\BP(\BaseOfD^\alg) \isoto \wt{W}^\BP  \backslash \wt{W}\slash \wt{W}^\BP
\end{equation}
where $\wt{W}^\BP := (N(L)\cap \BP(\CO_L))\slash \CT^0(\CO_L)$, and where $LP(R)=P(R\dbl z\dbr)$ and $L^+\BP(R)=\BP(R\dbl z\dbr)$ are the loop group, resp.\ the group of positive loops of $\BP$; see \cite[\S\,1.a]{PR2}, or \cite[\S4.5]{B-D}, \cite{Ngo-Polo} and \cite{Faltings03} when $\BP$ is constant. Let $\omega\in \wt{W}^\BP\backslash \wt{W}/\wt{W}^\BP$ and let $\BaseOfD_\omega$ be the fixed field in $\BaseOfD^\alg$ of $\{\gamma\in\Gal(\BaseOfD^\alg/\BaseOfD)\colon \gamma(\omega)=\omega\}$. There is a representative $g_\omega\in L\genericG(\BaseOfD_\omega)$ of $\omega$; see \cite[Example~4.12]{AH_Local}. The \emph{Schubert variety} $\CS(\omega)$ associated with $\omega$ is the ind-scheme theoretic closure of the $L^+\BP$-orbit of $g_\omega$ in $\SpaceFl_\BP\whtimes_{\BaseOfD}\BaseOfD_\omega$. It is a reduced projective variety over $\BaseOfD_\omega$. For further details see \cite{PR2} and \cite{Richarz}.

\section{Preliminaries}\label{Preliminaries}

Let $\BaseOfD$ be a finite field and $\BaseOfD\dbl z\dbr$ be the power series ring over $\BaseOfD$ in the variable $z$. We let $\BP$ be a smooth affine group scheme over $\BD:=\Spec\BaseOfD\dbl z\dbr$ with connected fibers. Set $\dot\BD:=\BF\dpl z\dpr$.

\begin{definition}\label{DefLoopGps}
The \emph{group of positive loops associated with $\BP$} is the infinite dimensional affine group scheme $L^+G$ over $\BaseOfD$ whose $R$-valued points for an $\BaseOfD$-algebra $R$ are 
\[
L^+\BP(R):=\BP(R\dbl z\dbr):=\BP(\BD_R):=\Hom_\BD(\BD_R,\BP)\,.
\]
The \emph{group of loops associated with $P$} is the $\fpqc$-sheaf of groups $LP$ over $\BaseOfD$ whose $R$-valued points for an $\BaseOfD$-algebra $R$ are 
\[
LP:=P(R\dpl z\dpr):=P(\dot{\BD}_R):=\Hom_{\dot\BD}(\dot\BD_R,P)\,,
\]
where we write $R\dpl z\dpr:=R\dbl z \dbr[\frac{1}{z}]$ and $\dot{\BD}_R:=\Spec R\dpl z\dpr$. It is representable by an ind-scheme of ind-finite type over $\BaseOfD$; see \cite[\S\,1.a]{PR2}, or \cite[\S4.5]{B-D}, \cite{Ngo-Polo}, \cite{Faltings03} when $\BP$ is constant.
Let $\scrH^1(\Spec \BaseOfD,L^+\BP)\,:=\,[\Spec \BaseOfD/L^+\BP]$ (respectively $\scrH^1(\Spec \BaseOfD,LP)\,:=\,[\Spec \BaseOfD/LP]$) denote the classifying space of $L^+\BP$-torsors (respectively $LP$-torsors). It is a stack fibered in groupoids over the category of $\BaseOfD$-schemes $S$, whose category $$\scrH^1(\Spec \BaseOfD,L^+\BP)(S)$$ consists of all $L^+\BP$-torsors (resp.\ $LP$-torsors) on $S$. The inclusion of sheaves $L^+\BP\subset LP$ gives rise to the natural 1-morphism 
\begin{equation}\label{EqLoopTorsor}
\scrH^1(\Spec \BaseOfD,L^+\BP)\longto \scrH^1(\Spec \BaseOfD,LP),~\CL_+\mapsto \CL\,.
\end{equation}
\end{definition}

\begin{definition}
The affine flag variety $\SpaceFl_\BP$ is defined to be the ind-scheme representing the $fpqc$-sheaf associated with the presheaf
$$
R\;\longmapsto\; L\genericG(R)/L^+\BP(R)\;=\;P\left(R\dpl z \dpr \right)/\BP \left(R\dbl z\dbr\right).
$$ 
on the category of $\BaseOfD$-algebras; compare Definition~\ref{DefLoopGps}.
\end{definition}

\begin{remark}\label{RemFlagisquasiproj}
 Note that $\SpaceFl_\BP$ is ind-quasi-projective over $\BaseOfD$ according to Pappas and Rapoport \cite[Theorem~1.4]{PR2}, and hence ind-separated and of ind-finite type over $\BaseOfD$. The quotient morphism $LP \to \SpaceFl_\BP$ admits sections locally for the \'etale topology.\forget{They proceed as follows. When $\BP = \SL_{r,\BD}$, the \fpqc-sheaf $\CF\ell_\BP$ is called the \emph{affine Grassmanian}. It is an inductive limit of projective schemes over $\BF$, that is, ind-projective over $\BF$; see \cite[Theorem~4.5.1]{B-D} or \cite{Faltings03, Ngo-Polo}.  By \cite[Proposition~1.3]{PR2} and \cite[Proposition~2.1]{AH_Global} there is a faithful representation $\BP ֒\to \SL_r$ with quasi-affine quotient. Pappas and Rapoport show in the proof of \cite[Theorem~1.4]{PR2} that $\CF\ell_\BP \to \CF\ell_{\SL_r}$ is a locally closed embedding, and moreover, if $\SL_r /\BP$ is affine, then $\CF\ell_\BP \to \CF\ell_{\SL_r}$ is even a closed embedding and $\CF\ell_\BP$ is ind-projective.} Moreover, if the fibers of $\BP$ over $\BD$ are geometrically connected, then $\CF\ell_\BP$ is ind-projective if and only if $\BP$ is a parahoric group scheme in the sense of Bruhat and Tits \cite[D\'efinition 5.2.6]{B-T}; see \cite[Theorem A]{Richarz13}.

\end{remark}

\noindent
Recall that for p-divisible groups $X$ and $Y$, one defines the following

\begin{enumerate}
\item
An \emph{isogeny} $f:X\to Y$ is a morphism which is an epimorphism as \fppf-sheaves and whose kernel is representable by a finite flat group scheme over $S$.
\item
A \emph{quasi-isogeny} is a global section $f$ of the Zariski sheaf $\Hom(X,Y)\otimes_\BZ \BQ$  such that $n\cdot f$ is an isogeny locally on $S$, for an integer $n\in\BZ$. 
\end{enumerate}

\noindent
Let us now recall the analogous definition over function fields, where we instead have \emph{local $\BP$-shtukas} and \emph{quasi-isogenies} between them.

\begin{definition}\label{localSht}
\begin{enumerate}
\item
A local $\BP$-shtuka over $S\in \Nilp_{\BaseOfD\dbl\zeta\dbr}$ is a pair $\ul \CL = (\CL_+,\tau)$ consisting of an $L^+\BP$-torsor $\CL_+$ on $S$ and an isomorphism of the associated loop group torsors $\tauLoc\colon  \hat{\sigma}^\ast \CL \to\CL$. 

\item
A \emph{quasi-isogeny} $f\colon\ul\CL\to\ul\CL'$ between two local $\BP$-shtukas $\ul{\CL}:=(\CL_+,\tau)$ and $\ul{\CL}':=(\CL_+' ,\tau')$ over $S$ is an isomorphism of the associated $LP$-torsors $f \colon  \CL \to \CL'$ such that the following diagram  

\[
\xymatrix {
\hat{\sigma}^\ast\CL \ar[r]^{\tau} \ar[d]_{\hat{\sigma}^\ast f} & \CL\ar[d]^f  \\
\hat{\sigma}^\ast\CL' \ar[r]^{\tau'}&  \CL' \;.
}
\]

becomes commutative.
\item
We denote by $\QIsog_S(\ul{\CL},\ul{\CL}')$ the set of quasi-isogenies between $\ul{\CL}$ and $\ul{\CL}'$ over $S$. 

\item
We let $Loc-\BP-Sht(S)$ denote the category of local $\BP$-shtukas over $S$ with quasi-isogenies as the set of morphisms. 

\end{enumerate}
\end{definition}

Let us recall that quasi-isogenies of p-divisible groups are rigid in the following sense. Let $X$ and $Y$ be p-divisible groups over $S$. Let $\ol S\to S$ be a nilpotent thickening, i.e. a closed immersion defined by a nilpotent sheaf of ideal. Then, the restriction $QIsog_S(X,Y)\to QIsog_{\ol S}(\ol X,\ol Y)$ between the set of quasi-isogenies is a bijection.  
\bigskip

\noindent
Likewise, local $\BP$-shtukas enjoy a similar rigidity property.

\bigskip

\begin{proposition}[Rigidity of quasi-isogenies for local $\BP$-shtukas] \label{PropRigidityLocal}
Let $S$ be a scheme in $\Nilp_{\BaseOfD\dbl\zeta\dbr}$ and
let $j \colon  \bar{S}\rightarrow S$ be a closed immersion defined by a sheaf of ideals $\CI$ which is locally nilpotent.
Let $\ul{\CL}$ and $\ul{\CL}'$ be two local $\BP$-shtukas over $S$. Then
$$
\QIsog_S(\ul{\CL}, \ul{\CL}') \longto \QIsog_{\bar{S}}(j^*\ul{\CL}, j^*\ul{\CL}') ,\quad f \mapsto j^*f
$$
is a bijection of sets.
\end{proposition}

\begin{proof} See \cite[Proposition 2.11]{AH_Local}.

\end{proof}

In the rest of this section we want to recall the notion of \emph{local boundedness condition}, introduced in \cite[Definition~4.8]{AH_Local}.\\

\noindent
Fix an algebraic closure $\BaseOfD\dpl\zeta\dpr^\alg$ of $\BaseOfD\dpl\zeta\dpr$. For a finite extensions of discrete valuation rings $R/\BaseOfD\dbl\zeta\dbr$ with $R\subset\BaseOfD\dpl\zeta\dpr^\alg$, we denote by $\kappa_R$ its residue field, and we let $\Nilp_R$ be the category of $R$-schemes on which $\zeta$ is locally nilpotent. We also set $\wh{\SpaceFl}_{\BP,R}:=\SpaceFl_\BP\whtimes_{\BaseOfD}\Spf R$ and $\wh{\SpaceFl}_\BP:=\wh{\SpaceFl}_{\BP,\BaseOfD\dbl\zeta\dbr}$. Before we can define (local) ``boundedness condition'', let us recall that $\wh{\SpaceFl}_{\BP,R}$ can be viewed as an unbounded Rapoport-Zink spaces for local $\BP$-shtukas. We explain this in the next section. For now, consider the following functor

\begin{eqnarray*}
\ul\CM:&(\Nilp_R)^o &\longto  \Sets  \hspace{7cm}\vspace{-2mm}\\ 
&\SSS &\longmapsto  \big\{\text{Isomorphism classes of }(\CL_+,\delta);\text{where: }\\  
& &~~~~~~~~~~~~~~~~~~~~~~~~~~~~-\CL_+~\text{is an $L^+\BP$-torsor over $\SSS$ and}\\ \nonumber
& &  ~~~~~~~~~~~~~~~~~~~~~~~~~~~~ -\text{a trivialization $\delta\colon  \CL \to LP_S$ of the}\\ \nonumber
& & ~~~~~~~~~~~~~~~~~~~~~~~~~~~~~~~~~\text{associated loop torsors}
\big\}. 
\end{eqnarray*}

\begin{proposition}\label{PropFlRepUnBoundedRZ}
The ind-scheme $\wh{\CF\ell}_{\BP,R}$ pro-represents the above functor.

\end{proposition}

\begin{proof}
In order to illustrate how the representablity works, here we briefly sketch the proof, and we refer the reader to \cite[Theorem~4.4.]{AH_Local} for further details. We assume that $R=\BF\dbl\zeta\dbr$. Consider a pair $(\CL_+,\delta)\in \ul\CM(S)$. Choose an \fppf-covering $S' \to S$ which trivializes $\CL_+$, then the morphism $\delta$ is given by an element $g' \in LP(S')$. The image of the element $g' \in LP(S')$ under  $LP(S')\to\wh{\CF\ell}_\BP (S')$ is independent of the choice of the trivialization, and since $(\CL_+,\delta)$ is defined over $S$, it descends to a point $x \in \wh{\CF\ell}_\BP$. 

Conversely let $x$ be in $\wh{\CF\ell}_\BP(S)$, for a scheme $S\in \Nilp_{\BF\dbl\zeta\dbr}$. The projection morphism $LP \to \CF\ell_\BP$ admits local sections for the \'etale topology by \cite[Theorem~1.4]{PR2}. Hence over an \'etale covering $S' \to S$ the point $x$ can be represented by an element $g'\in LP(S')$. We let $(\CL_+',\delta')=((L^+\BP)_{S'},g')$. It can be shown that it descends and gives $(\CL_+,\delta)$ over $S$.
\end{proof}

\noindent
Here we recall the definition of local boundedness condition from \cite[Definition~4.8]{AH_Local}.

\begin{definition}\label{DefLBC}

\begin{enumerate}
\item\label{DefBDLocal_A}
 For a finite extension of discrete valuation rings $\BaseOfD\dbl\zeta\dbr\subset R\subset\BaseOfD\dpl\zeta\dpr^\alg$ we consider closed ind-subschemes $\wh Z_R\subset\wh{\SpaceFl}_{\BP,R}$. We call two closed ind-subschemes $\wh Z_R\subset\wh{\SpaceFl}_{\BP,R}$ and $\wh Z'_{R'}\subset\wh{\SpaceFl}_{\BP,R'}$ \emph{equivalent} if there is a finite extension of discrete valuation rings $\BaseOfD\dbl\zeta\dbr\subset\wt R\subset\BaseOfD\dpl\zeta\dpr^\alg$ containing $R$ and $R'$ such that $\wh Z_R\whtimes_{\Spf R}\Spf\wt R \,=\,\wh Z'_{R'}\whtimes_{\Spf R'}\Spf\wt R$ as closed ind-subschemes of $\wh{\SpaceFl}_{\BP,\wt R}$.

\medskip\noindent

\item\label{DefBDLocal_B} Let $\wh Z=[\wh Z_R]$ be an equivalence class in the above sense. The \emph{reflex ring} $R_{\hat{Z}}$ is defined as the intersection of the fixed field of $\{\gamma\in\Aut_{\BaseOfD\dbl\zeta\dbr}(\BaseOfD\dpl\zeta\dpr^\alg)\colon \gamma(\wh Z)=\wh Z\,\}$ in $\BaseOfD\dpl\zeta\dpr^\alg$ with all the finite extensions $R\subset\BaseOfD\dpl\zeta\dpr^\alg$ of $\BaseOfD\dbl\zeta\dbr$ over which a representative $\wh Z_R$ of $\wh Z$ exists.
\item \label{DefBDLocal_C}
We define a (local) \emph{bound} to be an equivalence class $\wh Z:=[\wh Z_R]$ of closed ind-subschemes $\wh Z_R\subset\wh{\SpaceFl}_{\BP,R}$, such that \\

\begin{enumerate}

\item[-] all the ind-subschemes $\wh Z_R$ are stable under the left $L^+\BP$-action on $\SpaceFl_\BP$ and\\

\item[-]
 the special fibers $Z_R:=\wh Z_R\whtimes_{\Spf R}\Spec\kappa_R$ are quasi-compact subschemes of the ind scheme $\SpaceFl_\BP\whtimes_{\BaseOfD}\Spec\kappa_R$. 

\end{enumerate}

 Note that $Z_R$ arise by base change from a unique closed subscheme $Z\subset\SpaceFl_\BP\whtimes_\BaseOfD\kappa_{R_{\wh Z}}$. This is because the Galois descent for closed subschemes of $\SpaceFl_\BP$ is effective. We call $Z$ the \emph{special fiber} of the bound $\wh Z$. It is a projective scheme over $\kappa_{R_{\wh Z}}$ by~\cite[Remark 4.3]{AH_Local} and \cite[Lemma~5.4]{H-V}, which implies that every morphism from a quasi-compact scheme to an ind-projective ind-scheme factors through a projective subscheme.
\item \label{DefBDLocal_D}
Let $\wh Z$ be a bound with reflex ring $R_{\wh Z}$. Let $\CL_+$ and $\CL_+'$ be $L^+\BP$-torsors over a scheme $S$ in $\Nilp_{R_{\wh Z}}$ and let $\delta\colon \CL\isoto\CL'$ be an isomorphism of the associated $LP$-torsors. We consider an \'etale covering $S'\to S$ over which trivializations $\alpha\colon\CL_+\isoto(L^+\BP)_{S'}$ and $\alpha'\colon\CL_+'\isoto(L^+\BP)_{S'}$ exist. Then the automorphism $\alpha'\circ\delta\circ\alpha^{-1}$ of $(LG)_{S'}$ corresponds to a morphism $S'\to LG\whtimes_\BaseOfD\Spf R_{\wh Z}$. We say that $\delta$ is \emph{bounded by $\wh Z$} if for any such trivialization and for all finite extensions $R$ of $\BaseOfD\dbl\zeta\dbr$ over which a representative $\wh Z_R$ of $\wh Z$ exists the induced morphism $S'\whtimes_{R_{\wh Z}}\Spf R\to LP\whtimes_\BaseOfD\Spf R\to \wh{\SpaceFl}_{\BP,R}$ factors through $\wh Z_R$. Furthermore we say that a local $\BP$-shtuka $(\CL, \tauLoc)$ is \emph{bounded by $\wh Z$} if the isomorphism $\tauLoc^{-1}$ is bounded by $\wh Z$. Assume that $\wh Z=\CS(\omega)\whtimes_\BaseOfD\Spf \BaseOfD\dbl\zeta\dbr$ for a \emph{Schubert variety} $\CS(\omega)\subseteq \CF\ell_\BP$ , with $\omega\in \wt{W}$; see \cite{PR2}. Then we say that $\delta$ is \emph{bounded by $\omega$}.

\end{enumerate}
\end{definition}

\section{The Analogy between Shimura varieties and moduli of $\FG$-shtukas}\label{SectAnalogy}

In this chapter we discuss some aspects of the analogy between Shimura varieties (resp. Rapoport-Zink spaces) and their function fields counterparts, i.e. the moduli stacks of global $\FG$-shtukas (Rapoport-Zink spaces for local $\BP$-shtukas). In particular we discuss the analogy between the uniformization theory and local model theory for these moduli spaces. Furthermore, we prove the local model theorem in last subsection \ref{SubsectProof} of this chapter.

\subsection{Local Shimura data}\label{SubsectLocalShimuradata}
Recall that a Shimura data $(\BG,X,K)$ consists of
\begin{enumerate}

\item[-]
a reductive group $\BG$ over $\BQ$ with center $Z$,
\item[-]
$\BG(\BR)$-conjugacy class $X$ of homomorphisms $\BS\to G_\BR$ for the Deligne torus $\BS$, 
\item[-]
A compact open sub-group $K\subseteq \BG(\BA_f)$,  
\end{enumerate}
subject to certain conditions; see \cite{MilneShimura}. \\

Let us fix a prime number $p$ and write $K=K_p \cdot K^p$ for compact open subgroups $K_p\subseteq \BG(\BQ_p)$ and $K^p\subseteq \BG(\BA_\BQ^p)$. The above data determine a reflex field $E:=E(\BG,X,K)$, and the corresponding Shimura variety 

$$
 Sh_K(\BG,X)=\BG(\BQ)\backslash\bigl(X\times \BG(\BA_f)/K\bigr),
$$
\noindent  
which admits a canonical integral model $\scrS_K$ over \forget{$W= W(\ol\BF_p)$}reflex ring $\CO_E$\forget{ for $K=G(\BZ_p)K^p$}, for sufficiently small $K^p\subseteq \BG(\BA_\BQ^p)$; see \cite{Kis}. \\

Shimura varieties have local counterparts, which are called \emph{Rapoport-Zink spaces}. They arise from \emph{local Shimura data} and, roughly speaking, parametrize families of quasi-isogenies of $p$-divisible groups to a fixed one.  Let us explain it a bit further. 

\begin{definition}\label{DefLSD}
A \emph{local Shimura data}  is a tuple $(\CP, \{\mu\}, [b])$ consisting of

\begin{enumerate}

\item[-]
a smooth affine group scheme $\CP$ over $\BZ_p$ with a connected reductive generic fiber $G$ over $\BQ_p$,
\item[-]
a conjugacy class of a (minuscule) cocharacter $\mu : \BG_m \to G$,

\item[-]
a class $[b]$ in $B(G, \mu)$ of Kottwitz set of $\sigma$-conjugacy classes.

\end{enumerate}

\end{definition}

In particular the local Shimura data $(\CP, \{\mu\}, [b])$ determines the reflex field $E:=E_\mu$, which is the field of definition of the cocharacter $\mu$. We set $\CO=\CO_{E}$.\\
Assume that $G$ splits over a tamely ramified extension and $\CP$ is parahoric (i.e. the special fiber $\CP_s=\CP_{\BF_p}$ is connected). To such a data one associates a formal  scheme $\check{\CM}:=\check{\CM}(\CP, [b], \{\mu\})$ over $\CO$ (that up to some modifications is a moduli space for p-divisible groups with additional structures). This is called \emph{Rapoport-Zink space}, associated to the local Shimura
data $(\CP, [b], \{\mu\})$. The underlying scheme $\check{\CM}_{\red}$ is a union of affine Deligne-Lusztig varieties. For details we refer the reader to \cite{RV} and \cite{SchWei}. See also \cite{Kim} and Shen \cite{Shen} for generalizations to the Rapoport-Zink spaces of Hodge type and abelian type, respectively.

\forget{This space was constructed by Rapoport and Zink for Shimura varieties of PEL-type \cite{RZ} and further generalized to Shimura varieties of Hodge type by Kim \cite{Kim} and abelian type by Shen \cite{Shen}. }

\bigskip
\noindent
\subsection{Local Models for Rapoport-Zink spaces}\label{SubsectLocalModelsforR-Zspaces}
In \cite{RZ} the authors propose the following method to study local properties of Rapoport-Zink spaces.

\begin{definition}
Let $F/\BQ_p$ be a finite extension. A local model triple $(G,\{\mu\},\CP)$ over $F$ consists of the following data
\begin{enumerate}

\item[-]
$G$ is a reductive group over $F$,
\item[-]
$\{\mu\}$ is a conjugacy class of a cocharacter of $G$,
\item[-]
$\CP$ is a parahoric group scheme over $\CO_F$.

\end{enumerate}
 
 \end{definition}

We assume that $G$ splits over a tamely ramified extension. One can associate to a local model data $(G,\{\mu\},\CP)$, a variety $\scrA:=\scrA(G,\{\mu\},\CP)$ over $k:=\ol\kappa_E$ inside the affine flag variety $\CF\ell_{\CP,k}$, with an action of $\CP\otimes_{\CO_F} k$. More explicitly $\scrA$ is the union $\cup_{\omega\in Adm_{\CP}(\mu)} S(\omega)$ of affine Schubert varieties $S(\omega)$. Here $$Adm_{\CP}(\mu)=\{\omega\in \wt{W}^\CP\backslash\wt{W}/\wt{W}^\CP; \omega \preceq t^\mu  \}.$$

\begin{definition}
A local model $\LM:=\LM(G,\{\mu\},\CP)$ attached to a local model data $(G,\{\mu\},\CP)$ is a projective scheme $\LM$ over $\CO_E$, with generic fiber $\textbf{M}_\eta^{loc}$ (resp. special fiber $\textbf{M}_s^{loc}$) with an action of $\CP\otimes_{\CO_F}\CO_E$, subject to the following conditions

\begin{enumerate}

\item[-]
It is flat over $\CO_E$ with reduced special fiber,
\item[-]
There is a $\CP\otimes k$-equivariant isomrphism $\textbf{M}_s^{loc}\cong \scrA$, 
\item[-]
There is a $G_E$-equivariant isomrphism $\textbf{M}_\eta^{loc}\cong G_E/P_{\{\mu\}}$, 

\end{enumerate}
 
\end{definition}

Note that in particular all irreducible components of $\LM\otimes k$ are normal and Cohen-Macaulay, see \cite[Theorem~8.4]{PR2}.
\noindent
The existence and uniqueness of $\LM$ is known for $EL$ and $PEL$ cases. In general only the existence is known according to \cite{PZ} and uniqueness is not known; also compare \cite[Proposition~18.3.1]{SchWei}, where the authors use an alternative definition and then they serve uniqueness but not the existence! \\
\noindent
According to the local model theory there is 
a local model roof

\begin{equation}\label{nablaHRoof} 
\xygraph{
!{<0cm,0cm>;<1cm,0cm>:<0cm,1cm>::}
!{(0,0) }*+{\wt{\CM}}="a"
!{(-1.5,-1.5) }*+{\check{\CM}}="b"
!{(1.5,-1.5) }*+{\LM,}="c"
"a":_{\pi}"b" "a":^{\pi^{loc}}"c"
}  
\end{equation}
\noindent
where $\pi:\wt{\CM}\to\check{\CM}$ is a $\CP$-torsor and $\pi^{loc}$ is formally smooth of relative dimension $\dim \CP$. In particular for every $x\in\check{\CM}(k)$ there is a $y\in\LM(k)$ with $\wh{\CO}_{\check{\CM},x}\cong\wh{\CO}_{\LM,y}$.


\bigskip

\noindent
\subsection{Points of Shimura varieties mod p and uniformization theory}\label{Subsect Points of Shimura varieties mod p and uniformization theory}

As we mentioned above, the Rapoport-Zink spaces are local counterparts for Shimura varieties. In particular their $\ell$-adic cohomology is supposed to eventually realize the local Langlands correspondence, according to a conjecture of Kottwitz \cite{Rap94}.\\

 The geometry of Shimura varieties and Rapoport-Zink spaces are related through local model theory and \emph{uniformization theory}. The underlying facts which play crucial role here are as follows
 
 \begin{enumerate}
\item[(\textasteriskcentered)]
The deformation space of an abelian variety is completely ruled by associated crystal (according to Grothendieck-Messing theory), and \\
\item[(\textasteriskcentered\textasteriskcentered)]
one can pull back an abelian variety $\CA$ (resp. motive $\CM$) along a quasi-isogeny of $p$-divisible groups $\rho: \CX\to \CA[p^\infty]$ (associated crystalline realization). \\ 

\end{enumerate}

Let us assume that the local Shimura data $(\CP, [b_\phi], \{\mu\})$ arise from the integral model $\scrS_K$ for the Shimura variety corresponding to the Shimura data $(\BG,X,K)$ and a morphism $\phi: \FQ\to \CG_\BG$, from the quasi-motivic Galois gerb $\FQ$, see \cite{Kis}, to the neutral Galois gerb $\CG_\BG$. In this situation the Rapoport-Zink space is equipped with symmetries by a reductive group $I_\phi$. Now, regarding the fact $(\textasteriskcentered\textasteriskcentered)$, one can produce the uniformization map

$$
\coprod_\phi I\phi(\BQ) \backslash \check{\CM} (G, [b_\phi], \{\mu\}) \times G(\BA_f^p)/K \to \scrS_K
$$

\noindent
Furthermore, the uniformization theorem states that the above morphism induces an isomorphism after passing to the completion along certain subvarieties $\CT_{\phi,K^p}\subseteq\scrS_K$. For the proof for the Shimura varieties of PEL-type see \cite{RZ}[Chapter~6], and for generalizations to Shimura varieties of Hode type and abelian type see \cite{Kim} and \cite{Shen}, respectively.

\begin{remark}\label{RemL-RI} Note that the uniformization theory provides a geometric meaning for the Langlands-Rapoport description of the $\ol\BF_p$-points of Shimura varieties, formulated in \cite{LR}. 

\end{remark}


\bigskip

\noindent
Let us now move to the analogous picture over function fields.\\

\subsection{The analogous picture over function fields}\label{SubsecTheAnalogousPic}
\noindent
Here the shimura data $(\BG,X,K)$, would be replace by a tuple $(\FG, \ul{\hat{Z}}, H)$, which is called \emph{$\nabla\scrH$-data}.

\begin{definition}\label{DefNablaH-data}
A  $\nabla\scrH$-data is a tuple $(\FG, \ul{\hat{Z}}, H)$ consisting of
\begin{enumerate}

\item[-] a smooth affine group scheme $\FG$ over a smooth projective curve $C$ over $\BF_q$,

\item[-] an $n$-tuple of (local) bounds $\ul{\hat{Z}}:=(\hat{Z}_{\nu_i})_{i=1\dots n}$, in the sense of Definition \ref{DefLBC}, at the fixed characteristic places $\nu_i\in C$, 

\item[-] a compact open subgroup $H\subseteq \FG(\BA_C^\ul\nu)$.

\end{enumerate}

\end{definition}

To such a data we associate a moduli stack
$\nabla_n^{H,\ul{\hat{Z}}}\scrH^1(C,\FG)^\ul\nu$ parametrizing global $\FG$-shtukas with level $H$-structure which are in addition bounded by $\ul{\hat{Z}}$. Let us explain this assignment a bit further. For this purpose, we first recall the definition of the moduli of $\FG$-shtukas (without imposing boundedness conditions).

\begin{definition}\label{Global Sht}

A global $\FG$-shtuka $\ul\CG$ over an $\BF_q$-scheme $S$ is a tuple $(\CG,\ul\charsect,\tauGlob)$ consisting of a $\FG$-bundle $\CG$ over $C_S$, an $n$-tuple of (characteristic) sections $\ul\charsect$, and an isomorphism $\tauGlob\colon  \s \CG|_{C_S\setminus \Gamma_{\ul\charsect}}\isoto \CG|_{C_S\setminus  \Gamma_{\ul\charsect}}$. We let $\nabla_n\scrH^1(C,\FG)$ denote the  stack whose $S$-points is the category of $\FG$-shtukas over $S$. Clearly there is a natural projection $\nabla_n\scrH^1(C,\FG)\to C^n$. Set $$\nabla_n\scrH^1(C,\FG)^\ul\nu:=\nabla_n\scrH^1(C,\FG)\times_{C^n}\Spf A_\ul\nu .$$  

\end{definition}

The resulting moduli stack $\nabla_n\scrH^1(C,\FG)^\ul\nu$ is an ind-algebraic stack. See \cite{AH_Global} for further explanation regarding the ind-algebraic structure. Furthermore, this moduli stack can be considered as a moduli for (C-)motives with $\FG$-structures. More precisely, for a connected test scheme $S$ in $\Nilp_{A_\ul\nu}$ there are \'etale 

$$
\omega^\ul\nu(-): \nabla\scrH^1(C,\CG)^\ul\nu(S)\to  Funct^\otimes (\Rep \FG ,\FM od_{\BA_C^\ul\nu[\pi_1(S,\ol s)]})
$$
\noindent
and crystalline
$$
\omega_{\nu_i}(-): \nabla\scrH^1(C,\FG)^\ul\nu(S)\to Loc-\BP_{\nu_i}-Sht(S) 
$$ 
\noindent
realization functors. The target of the first functor is  the category of tensor functors from the category of representations $\Rep \FG$ to the category $\FM od_{\BA_C^\ul\nu[\pi_1(S,\ol s)]}$ of modules over $\BA_C^\ul\nu[\pi_1(S,\ol s)]$, where $\pi_1(S,\ol s)$ denotes the algebraic fundamental group of $S$, with a geometric base point $\ol s\in S$. See \cite{AH_Global} or \cite{AH_LR} for the constructions and properties of these functors. Note that the set of morphisms of both categories $\nabla\scrH^1(C,\FG)^\ul\nu(S)$ and $Loc-\BP_{\nu_i}-Sht(S)$ can be enlarged to the set of quasi-isogenies.

Using tannakian formalism, we can equip the moduli stack $\nabla\scrH^1(C,\FG)^\ul\nu$ of global $\FG$-shtukas with $H$-level structure, for a compact open subgroup $H\subseteq \FG(\BA_Q^{\ul \nu})$; details are explained in \cite[Chapter~6]{AH_Global}. Here we only recall the definition.
  
\begin{definition}[$H$-level structure]\label{DefLevelStr}
Assume that $S\in \Nilp_{\wh A_\ul\nu}$ is connected and fix a geometric point $\ol s$ of $S$. Let $\pi_1(S,\bar{s})$ denote the algebraic fundamental group of $S$. For a global $\FG$-shtuka $\ul\CG$ over $S$ let us consider the set of isomorphisms of tensor functors $\Isom^{\otimes}(\omega^\ul\nu(\ul{\CG})(-),\omega^\circ(-))$, where $\omega^\circ\colon \Rep_{\BA_C^{\ul \nu}}\FG \to \FM od_{\BA_C^{\ul \nu}}$ denote the neutral fiber functor. The set $\Isom^{\otimes}(\omega^\ul\nu(\ul{\CG})(-),\omega^\circ(-))$ admits an action of $\FG(\BA_C^{\ul \nu})\times\pi_1(S,\bar{s})$ where $\FG(\BA_C^{\ul \nu})$ acts through $\omega^\circ(-)$ by tannakian formalism and $\pi_1(S,\bar{s})$ acts through $\omega^\ul\nu(\ul{\CG})(-)$. For a compact open subgroup $H\subseteq \FG(\BA_C^{\ul \nu})$ we define a \emph{rational $H$-level structure} $\bar\gamma$ on a global $\FG$-shtuka $\ul \CG$ over $S\in\Nilp_{\wh A_\ul\nu}$ to be a $\pi_1(S,\bar{s})$-invariant $H$-orbit $\bar\gamma=H\gamma$ in $\Isom^{\otimes}(\omega^\ul\nu(\ul{\CG})(-),\omega^\circ(-))$.

\end{definition}
\noindent
Now we define $\nabla_n^{H,\ul{\wh Z}}\scrH^1(C,\FG)^{\ul \nu}$:

\begin{definition}\label{DefBoundedNablaH}

Let $(\FG, \ul{\hat{Z}}, H)$ be a $\nabla\scrH$-data in the above sense. We denote by $$\nabla_n^{H,\ul{\wh Z}}\scrH^1(C,\FG)^{\ul \nu}$$ the category fibered in groupoids, whose category of $S$-valued points $\nabla_n^{H,\ul{\wh Z}}\scrH^1(C,\FG)^{\ul \nu}(S)$ parametrizes tuples $(\ul \CG,\bar\gamma)$, consisting of a global $\FG$-shtuka $\ul\CG$ together with a rational $H$-level structure $\bar\gamma$ such that the corresponding local $\BP_{\nu_i}$-shtuka $\omega_{\nu_i}(\ul\CG)$ is bounded by $\hat{Z}_{\nu_i}$ for every $1\leq i\leq n$. 

\end{definition}

\noindent
This finally establishes the assignment

$$
(\FG,\ul{\wh Z},H)\mapsto \nabla_n^{H,\ul{\wh Z}}\scrH^1(C,\FG)^{\ul \nu}.
$$

\noindent
The moduli stack $\nabla_n^{H,\ul{\wh Z}}\scrH^1(C,\FG)^{\ul \nu}$ is a formal algebraic stack according to \cite{AH_Global}, and lies over a completed fiber product of reflex rings $R_{\ul{\wh Z}}:=R_{\wh Z_{\nu_1}}\hat{\times}\dots\hat{\times} R_{\wh Z_{\nu_n}}$. Note that the above assignment is functorial with respect to the morphisms between two $\nabla\scrH$-data. The functoriality with respect to $H$, allows to equip the inverse limit $$\nabla_n^{\ast,\ul{\wh Z}}\scrH^1(C,\FG)^{\ul \nu}:=\invlim[H]\nabla_n^{H,\ul{\wh Z}}\scrH^1(C,\FG)^{\ul \nu},$$ and hence it's cohomology, with Hecke operation; see \cite{AH_LR}[Definition~3.4]. \\

\noindent
The category of $\GL_n$-shtukas $\nabla\scrH^1(C,\GL_n)(\ol\BF_q)$ with quasi-isogenies as the set of morphisms is equivalent with the category of $C$-motives $\mot(\ol\BF_q)$. The latter category can  be defined over a general field $L/\BF_q$ and generalizes the category of $t$-motives, see \cite{Anderson} and \cite{Taelman}. Furthermore $\mot(\ol\BF_q)$ is a semi-simple tannakian tensor category, equipped with a fiber functor 
$$\omega(-):\mot(\ol\BF_q)\to \ol\BF_q.Q-Vector~Spaces,
$$
see \cite{AH_LR} and also \cite{A_CMot}. Note that this is despite the complicated situation over number fields, where the corresponding fact only relies on the Tate conjecture and Grothendieck’s standard conjectures. This allows one to define the corresponding motivic groupoid $\FP:=\mot(\ol\BF_q)(\omega)$ and identify  a morphism $\phi: \FP \to \CG_\FG$ with a $\FG$-shtuka $\ul\CG$ in $\nabla\scrH^1(C,\FG)(\ol\BF_q)$; see \cite{AH_LR}. \\

\noindent
\subsubsection{Rapoport-Zink spaces for local $\BP$-shtukas and their local models}\label{Rapoport-Zink spaces for local Shtukas}
In analgy with the Shimura variety side we define
\begin{definition}\label{DefLocalNablaHdata}
\emph{A local $\nabla\scrH$-data} is a tuple $(\BP,\hat{Z}, b)$ consisting of

\begin{enumerate}
\item[-]
A smooth affine group scheme over $\BP$ with connected reductive generic fiber $P$,
\item[-]
A local bound $\hat{Z}$ in the sense of Definition \ref{DefLBC}.
\item[-]
A $\sigma$-conjuagacy class of an element $b\in P(\ol\BF \dpl z \dpr)$. 
\end{enumerate}

\end{definition} 
 
To a tuple $(\BP,\hat{Z}, b)$ of local $\nabla\scrH$-data one may associate a formal scheme $\check{\CM}(\BP,\hat{Z}, b)$ which is a moduli space for local $\BP$-shtukas together with a quasi-isogeny to a fixed local $\BP$-shtuka $\ul\BL$, determined by the local $\nabla\scrH$-data. In analogy with number fields, they are called Rapoport-Zink spaces (for local $\BP$-shtukas). These moduli spaces were first introduced and studied in \cite{H-V} for the case where $\BP$ is a constant split reductive group over $\BD$, and then generalized to the case where $\BP$ is a smooth affine group scheme over $\BD$ with connected reductive generic fiber in \cite{AH_Local}. Let us briefly recall the construction of these formal schemes.

Let $\ul\BL$ be a local $\BP$-shtuka over $\BF$. The bound $\hat{Z}$ determines the reflex ring $R_{\hat{Z}}$. Consider the following functor

\begin{eqnarray*}
\ul{\check{\CM}}_{\ul\BL}:&(\Nilp_{\check{R}_{\hat{Z}}})^o &\longto  \Sets  \hspace{7cm}\vspace{-2mm}\\
&\SSS &\longmapsto  \big\{\text{Isomorphism classes of }(\ul\CL,\delta); \text{where: }\\
& &~~~~~~~~~~~~~~~~~~~~~~~~~~~~-\ul\CL~\text{is a local $\BP$-shtuka over $\SSS$ and}\\ 
& &  ~~~~~~~~~~~~~~~~~~~~~~~~~~~~ -\text{$\ol\delta\colon  \ul\CL_{\ol S} \to \ul\BL_\ol S$ is a quasi-isogeny}
\big\}. 
\end{eqnarray*}
\noindent
Here $\ol S$ is the closed subscheme of $S$ defined by $\zeta=0$. By rigidity property of quasi-isogenies Proposition \ref{PropRigidityLocal}, this functor is equivalent with the functor introduced in  Proposition \ref{PropFlRepUnBoundedRZ}, and therefore is representable by $\wh\CF\ell_{\BP, R_{\hat{Z}}}$. Now we define the following sub-functor

\noindent
\begin{definition}
 Let $\hat{Z}=([R, \hat{Z}_R])$ be a bound with reflex field $R_{\hat{Z}}$. 

\begin{enumerate} 
\item
We define \emph{the Rapoport-Zink space for (bounded) local $\BP$-shtukas}, as the space given by the following functor of points

\begin{eqnarray*}
\ul{\check{\CM}}_{\ul\BL}^{\hat{Z}}:&(\Nilp_{\check{R}_{\hat{Z}}})^o &\longto  \Sets  \hspace{7cm}\vspace{-2mm}\\
&\SSS &\longmapsto  \big\{\text{Isomorphism classes of }(\ul\CL,\delta)\colon\;\text{where: }\\
& &~~~~~~~~~~~~~~~~~~~~~~~~~~~~-\ul\CL~\text{is a local $\BP$-shtuka}\\
& & ~~~~~~~~~~~~~~~~~~~~~~~~~~~~~~~~\text{ over $\SSS$ bounded by $\hat{Z}$ and}\\ 
& &  ~~~~~~~~~~~~~~~~~~~~~~~~~~~~ -\text{$\ol\delta\colon  \ul\CL_{\ol S} \to \ul\BL_\ol S$ a quasi-isogeny}
\big\}. 
\end{eqnarray*}

\item
Let $Z$ denote the
special fiber of $\hat{Z}$ over $\kappa$. We define the associated affine Deligne-Lusztig variety as the reduced
closed ind-subscheme $X_{Z}(b) \subseteq \CF\ell_G$ whose $K$-valued points (for any field extension $K$ of $\BF$) are
given by
$$
X_{Z}(b)(K) :=\{ 
g\in \CF\ell_\BP(K): g^{-1}b \sigma^\ast(g) \in Z(K)\}
$$

\end{enumerate}
\end{definition}

The following theorem ensures the representability of the above functor by a formal scheme locally formally of finite type.

\begin{theorem}\label{ThmRZisLFF}
The functor $\ul{\check{\CM}}_{\ul\BL}^{\hat{Z}}$ is ind-representable
by a formal scheme over $\Spf \check{R}_{\hat{Z}}$ which is locally formally of finite type and separated. It is an ind-closed ind-subscheme of $\CF\ell_\BP \wh \times_{\BF_q} \Spf \check{R}_{\hat{Z}}$. Its underlying reduced subscheme equals $X_{Z}(b)$,
which is a scheme locally of finite type and separated over $\BF$, all of whose irreducible components
are projective.
The formal scheme representing it
is called a Rapoport-Zink space for bounded local $\BP$-shtukas.

\end{theorem}

\begin{proof}

See \cite[Theorem~4.18]{AH_Local}.

\end{proof}

The data $(\BP,\hat{Z}, b)$ determines the reflex ring $R_{\hat{Z}}$, see Definition \ref{DefLBC}, and a local $\BP$-shtuka $\ul\BL:=(L^+\BP,b\hat{\sigma})$, and thus we may establish the assignment

$$
(\BP,\hat{Z}, b)\mapsto \check{\CM}(\BP,\hat{Z}, b):=\ul{\check{\CM}}_{\ul\BL}^{\hat{Z}}.
$$

\noindent
Bellow, in analogy with Shimura variety side, we discuss the local model theory and uniformization theory for $\FG$-shtukas. Note here that these theories again rely on the analogs of the crucial facts (\textasteriskcentered) and (\textasteriskcentered\textasteriskcentered), where have been established in \cite[Proposition 5.7 and Theorem 5.10]{AH_Local}. In particular, in analogy with the Serre-Tate theorem, one can prove that for a nilpotent thickening $\ol S\to S$, and a global $\FG$-shtuka $\ol{\ul\CG}$ over $\ol S$, the following natural morphism

$$
Defo_{S/\ol S}(\ol{\ul\CG})\to\prod_{\nu_i} Defo_{S/\ol S} (\omega_{\nu_i}(\ol{\ul\CG})) 
$$
between deformation spaces is an isomorphism.
This suggests that one can study the local properties of Rapoport-Zink spaces for local $\BP$-shtukas by constructing so called local model diagram. 

\noindent
We let $\wt{\CM}_{\ul\BL}^{\hat{Z}}$ be the space associated to the following functor of points

\begin{eqnarray}\label{EqRecBd}
\wt{\CM}_{\ul\BL}^{\hat{Z}}:&(\Nilp_{\check{R}_{\hat{Z}}})^o &\longto \Sets  \hspace{7cm}\vspace{-2mm}\nonumber\\
&\SSS &\longmapsto \big\{ (\ul\CL:=(\CL_+,\tau),\delta, \gamma); \text{consisting of}\\ & & ~~~~~~~~~~~~~~~~~~~~~~~~~~~~-(\ul\CL,\delta) \in \check{\CM}_{\ul\BL}^{\hat{Z}} \text{and}\nonumber\\
& & ~~~~~~~~~~~~~~~~~~~~~~~~~~~~-\text{ a trivialization}~\gamma:\hat{\sigma}^\ast\CL_+\tilde{\to}L^+\BP
\big\}\nonumber
\end{eqnarray}
\noindent
Bellow we state the local model theorem for Rapoport-Zink spaces for $\BP$-shtukas. We postpone the proof to \ref{SubsectProof}.
\begin{theorem}\label{ThmRapoportZinkLocalModel}
Let $(\BP,\hat{Z},b)$ be a local $\nabla\scrH$-data and set $\check{\CM}_{\ul\BL}^{\hat{Z}}:=\check{\CM}(\BP,\hat{Z},b)$. There is a roof of morphism
\begin{equation}\label{nablaHRoof} 
\xygraph{
!{<0cm,0cm>;<1cm,0cm>:<0cm,1cm>::}
!{(0,0) }*+{\wt{\CM}_{\ul\BL}^{\hat{Z}}}="a"
!{(-1.5,-1.5) }*+{\check{\CM}_{\ul\BL}^{\hat{Z}}}="b"
!{(1.5,-1.5) }*+{\wh{Z},}="c"
"a":_{\pi}"b" "a":^{\pi^{loc}}"c"
}  
\end{equation}

satisfying the following properties

\begin{enumerate}
\item
the morphism $\pi^{loc}$ is formally smooth and
\item
the $L^+\BP$-torsor $\pi: \wt{\CM}_{\ul\BL}^{\hat{Z}}\to \check{\CM}_{\ul\BL}^{\hat{Z}}$ admits a section $s'$ locally for the \'etale topology on $\check{\CM}_{\ul\BL}^{\hat{Z}}$ such that $\pi^{loc}\circ s'$ is formally \'etale.
\end{enumerate}

\end{theorem}

\noindent
\subsubsection{Points of moduli of $\FG$-shtukas mod $\nu$ and uniformization theory}\label{SubsectPointsOfNablaH} 
Let $(\FG,\ul{\wh Z},H)$ be a $\nabla\scrH$-data. We consider the category of $C$-motives $\mot(\ol\BF_q)$. It is a semi-simple tannakian category with fiber functor $\omega: \mot(\ol\BF_q)\to \ol\BF_q(C)$-vector spaces. Furthermore it is equipped with crystalline realization $\omega_{\nu_i}$ (at the characteristic place $\nu_i$) and \'etale realization functor $\omega^{\ul\nu}$ (away from characteristic places $\nu_i$). For a detailed account on the construction of the category $\mot(\ol\BF_q)$, its properties and realization functors, see \cite{A_CMot}. Let $\FP$ denote the corresponding groupoid $\mot(\ol\BF_q)(\omega)$. Recall that giving an admissible morphism $\phi:\FP\to \CG_G$ is equivalent with fixing a global $\FG$-shtuka $\ul\CG$ over an algebraically closed field $\BaseFldInSectUnif$. So let us fix a global $\FG$-shtuka  $(\ul{\CG}_0,\bar\gamma_0)$  with characteristic $\ul\nu$, bounded by $\ul{\hat{Z}}$, with $H$-level structure $\bar\gamma_0=H\gamma_0\in H\backslash\Isom^{\otimes}(\omega^\ul\nu({\ul{\CG}_0}),\omega^\circ)$. It defines a point in $\nabla_n^{H,\ul{\hat{Z}}}\scrH^1(C,\FG)^{\ul\nu}(\BaseFldInSectUnif)$. We fix a representative $\gamma_0\in\Isom^{\otimes}(\omega^\ul\nu(\ul{\CG}_0),\omega^\circ)$ of $\bar\gamma_0$. Let $I_{\ul\CG_0}(Q)$ (equivalently $I_{\phi_0}(Q)$) be the group $\QIsog_\BaseFldInSectUnif(\ul\CG_0)$ of quasi-isogenies of $\ul\CG_0$. Let $(\ul{\BL}_i)_{i=1\ldots n}:=(\omega_{\nu_i}(\ul\CG_0))_i$ denote the associated tuple of local $\BP_{\nu_i}$-shtukas over $\BaseFldInSectUnif$. Let us fix a trivialization $\ul\BL_i=(L^+\BP, b_i)$. Let $\check{\CM}(\BP_{\nu_i},\wh{Z}_{\nu_i},b_i)$ denote the corresponding Rapoport-Zink space. By Theorem~\ref{ThmRZisLFF} the product 
$$
\prod_i \check{\CM}(\BP_{\nu_i},\wh{Z}_{\nu_i},b_i):=\check{\CM}(\BP_{\nu_1},\wh{Z}_{\nu_1},b_1)\wh\times_\BaseFldInSectUnif\ldots\wh \times_\BaseFldInSectUnif\check{\CM}(\BP_{\nu_n},\wh{Z}_{\nu_n},b_n)
$$ 
is a formal scheme locally formally of finite type over $R_{\ul{\wh Z}}:=R_{\wh Z_{\nu_1}}\wh\times\dots \wh\times R_{\wh Z_{\nu_n}}$. Note that the group $J_{\ul{\BL}_i}(Q_{\nu_i})=\QIsog_\BaseFldInSectUnif(\ul{\BL}_i)$ of quasi-isogenies of $\ul\BL_i$ over $\BaseFldInSectUnif$ acts naturally on $\check{\CM}(\BP_{\nu_i},\wh{Z}_{\nu_i},b_i)$. In particular we see that the group $I_{\ul\CG_0}(Q)$ acts on $\prod_i \check{\CM}(\BP_{\nu_i},\wh{Z}_{\nu_i},b_i)$ via the natural morphism 
\begin{equation}\label{EqI(Q)}
I_{\ul\CG_0}(Q)\longto \prod_i J_{\ul{\BL}_i}(Q_{\nu_i}),\es \alpha\mapsto \bigl(\omega_{\nu_i}(\alpha)\bigr)_i.
\end{equation}


The uniformization theorem \cite[Theorem~7.4]{AH_Global} for the moduli of $\FG$-shtukas, states the following

\begin{enumerate}
\item \label{Uniformization2_A} 
There is an $I_{\ul\CG_0}(Q)$-invariant morphism
$$
\Theta'\colon  \prod_i \check{\CM}(\BP_{\nu_i},\wh{Z}_{\nu_i},b_i)\times \FG(\BA_Q^{\ul\nu}) {\slash}H\longto \nabla_n^{H,\ul{\hat{Z}}}\scrH^1(C,\FG)^{\ul\nu}\whtimes_{\BF_{\ul\nu}}\Spec\BaseFldInSectUnif.
$$
Furthermore, this morphism factors through a morphism 
$$
\Theta\colon  I_{\ul\CG_0}(Q) \big{\backslash}\prod_i \check{\CM}(\BP_{\nu_i},\wh{Z}_{\nu_i},b_i)\times \FG(\BA_Q^{\ul\nu}) {\slash}H\longto \nabla_n^{H,\ul{\hat{Z}}}\scrH^1(C,\FG)^{\ul\nu}\whtimes_{\BF_{\ul\nu}}\Spec\BaseFldInSectUnif
$$
of ind-algebraic stacks over $R_{\ul{\wh{Z}}}$ which is ind-proper and formally \'etale. Note that both morphisms are compatible with the action of $\FG(\BA_Q^{\ul\nu})$ which acts through Hecke-corres\-ponden\-ces on source and target, see  \cite[Remark~7.5]{AH_Global}.

\item \label{Uniformization2_B}
Let $\{T_j\}$ be a set of representatives of $I_{\ul\CG_0}(Q)$-orbits of the irreducible components of 
$$
\prod_i \check{\CM}(\BP_{\nu_i},\wh{Z}_{\nu_i},b_i) \times \FG(\BA_Q^{\ul\nu}) {\slash}H.
$$ 
Then the image $\Theta'(T_j)$ of $T_j$ under $\Theta'$ is closed and each $\Theta'(T_j)$ intersects only finitely many others. Let $\CZ$ denote the union of the $\Theta'(T_j)$ and let $\nabla_n^{H,\ul{\hat{Z}}}\scrH^1(C,\FG)^{\ul\nu}_{/\CZ}$ be the formal completion of $\nabla_n^{H,\ul{\hat{Z}}}\scrH^1(C,\FG)^{\ul\nu}\whtimes_{\BF_{\ul\nu}}\Spec\BaseFldInSectUnif$ along $\CZ$. Then $\Theta$ induces an isomorphism of formal algebraic stacks over $R_{\ul{\wh Z}}$
$$
\Theta_{\CZ}\colon  I_{\ul\CG_0}(Q) \big{\backslash}\prod_i \check{\CM}(\BP_{\nu_i},\wh{Z}_{\nu_i},b_i)\times \FG(\BA_Q^{\ul\nu}){\slash}H\;\isoto\; \nabla_n^{H,\ul{\hat{Z}}}\scrH^1(C,\FG)^{\ul\nu}_{/\CZ}.
$$
\end{enumerate}

\begin{remark}
One can formulate a function field analog of Langlands-Rapoport conjecture, which describes the points of the special fiber of $\nabla_n^{H,\ul{\hat{Z}}}\scrH^1(C,\FG)^{\ul\nu}_{/\CZ}$ in terms of group theoretic data; see \cite{AH_LR}. See also Remark \ref{RemL-RI}. 
\end{remark}

We summarize the above discussion in the following table. We then prove   Theorem~\ref{ThmRapoportZinkLocalModel} in the next subsection.\\

\begin{tabular}{ |p{8cm}|p{8cm}|p{3cm}|p{3cm}|  }
 \hline
 Number Fields & Function Fields\\
 \hline
 \hline
 The group $\BG$ over $\BQ$ & The group $\FG$ over $C$\\
 \hline
 characteristic p & characteristic $\ul\nu=\{\nu_i\}$\\
 \hline
 $G_p:=\BG\times_\BQ \BQ_p$ & $\BP_{\nu_i}$\\
\hline
 $\BS\to \BG_\BR$& n-tuple of boundedness conditions $\ul{\hat{Z}}$ \\
\hline
A compact open subgroup $K\subseteq \BG(\BA_\BQ)$& A compact open subgroup $H\subseteq \BG(\BA_C)$\\
\hline
 Shimura data $(\BG,X,K)$   & $\nabla\scrH$-data $(\FG,\ul{\hat{Z}},H)$ \\ 
 \hline
reflex ring $\CO_E$ of the reflex field $E=E(G,X,K)$ & reflex ring $R_{\ul{\hat{Z}}}$\\
\hline
The canonical integral model $\scrS_K$ & Moduli stack $\nabla_n^{H,\ul{\hat{Z}}}\scrH^1(C,\FG)^\ul\nu$ \\
\hline
Local Shimura data $(\CP,\{\mu\},[b])$    & Local $\nabla\scrH$-data $(\BP, \hat{Z}, [b])$  \\
\hline
 p-divisible groups and (iso-)crystals (with additional structure) & Local ($\BP$-)Shtukas \\
\hline
Rapoport-Zink space $\check{\CM}(\CP,\{\mu\},[b])$ over the reflex ring $\CO_{E_\mu}$& Rapoport-Zink space $\check{\CM}(\BP, \hat{Z}, [b])$ over the reflex ring $R_{\hat{Z}}$\\
\hline
The local model $\LM$ & The scheme $\hat{Z}$\\
\hline
The local Model diagram 
\xygraph{
!{<0cm,0cm>;<1cm,0cm>:<0cm,1cm>::}
!{(0,0) }*+{\wt{\CM}(G,\{\mu\},[b])}="a"
!{(-1.5,-1.5) }*+{\check{\CM}(G,\{\mu\},[b])}="b"
!{(1.5,-1.5) }*+{\textbf{M}^{loc},}="c"
"a":_{\pi}"b" "a":^{\pi^{loc}}"c"
}  
& The local Model diagram
\xygraph{
!{<0cm,0cm>;<1cm,0cm>:<0cm,1cm>::}
!{(0,0) }*+{\wt{\CM}(\BP,\hat{Z},[b])}="a"
!{(-1.5,-1.5) }*+{\check{\CM}(\BP,\hat{Z},[b])}="b"
!{(1.5,-1.5) }*+{\wh{Z},}="c"
"a":_{\pi}"b" "a":^{\pi^{loc}}"c"
}  
\\
\hline 
 The category of motives $Mot(\ol\BF_q)$ with realization functors $\omega_\ell(-)$ and $\omega_p(-)$ & The category of $C$-motives $\mot(\ol\BF_q)$ with realization functors $\omega^\ul\nu(-)$ and $\omega_{\nu_i}(-)$\\
\hline
fiber functor $\omega(-): Mot({\BF}_q)\to \ol\BQ$-vect. sp. (conjectural) & The fiber functor $\omega:\mot(\ol{\BF}_q)\to\ol Q$-vect. sp. \\
 \hline
quasi-motivic galois gerb $\FQ$ &  The motivic groupoid $\FP:=\mot(\BF)(\omega)$ \\
 \hline
The uniformization map $$
\coprod_\phi I\phi(\BQ) \backslash \check{\CM} (G, [b_\phi], \{\mu\}) \times G(\BA_f^p)/K 
$$
$$
\Theta\downarrow
$$
$$
\scrS_K
$$
& The uniformization map $$
\coprod_\phi I_\phi(Q) \big{\backslash}\prod_i \check{\CM}(\BP_{\nu_i},\wh{Z}_{\nu_i},b_i)\times \FG(\BA_Q^{\ul\nu}) {\slash}H
$$
$$
\Theta\downarrow
$$
$$
 \nabla_n^{H,\ul{\hat{Z}}}\scrH^1(C,\FG)^{\ul\nu}\whtimes_{\BF_{\ul\nu}}\Spec\BaseFldInSectUnif
$$\\
\hline
 \multicolumn{2}{|c|}{The analogy between Shimura varieties and moduli of G-Shtukas} \\
 \hline
 
\end{tabular}

\newpage
\subsubsection{Proof of the local model theorem for Rapoport-Zink spaces for local $\BP$-shtukas}\label{SubsectProof}

\begin{proof} of Theorem~\ref{ThmRapoportZinkLocalModel}. Sending the tuple $(\ul\CL:=(\CL_+, \tau_{\CL}),\delta,\gamma)$ to $(\CL_+, \gamma\circ \tau_{\CL}^{-1})$ defines a map ${\wt{\CM}_{\ul\BL}^{\hat{Z}}} \to \wh{\CF\ell}_{G,\check{R}_{\hat{Z}}}$. As the local $\BP$-shtuka $\ul\CL$ is bounded by $\hat{Z}$, this morphism factors through $\hat{Z}$; see Definition \ref{EqRecBd}. This defines the map $\pi^{loc}:{\wt{\CM}_{\ul\BL}^{\hat{Z}}} \to \hat{Z}$.\\
Take a closed immersion $i: S_0\to S$ defined by a nilpotent sheaf of ideals $I$. Since $I$ is nilpotent, there is a morphism $j:S\to S_0$ such that the $q$-Frobenius $\sigma_S$ factors as follows

\[
\xymatrix {
S\ar[r]^j \ar@/^2pc/[rr]^{\sigma_S} & S_0 \ar[r]^i& S.
}
\]
\noindent
Let $(\CL_{0+}, \tau_{\CL_0},\delta_0,\gamma_0)$ be a point in $\wt{\CM}_{\ul\BL}^{\hat{Z}}(S_0)$ and assume that it maps to $(\CL_{0+},g_0)$ under $\pi^{loc}$. Furthermore assume that $(\CL_+,g:\CL\tilde{\to} (LP)_S)$  lifts $(\CL_{0+},g_0)$ over $S$, i.e. $i^\ast \CL=\CL_0$ and $i^\ast g= g_0= \gamma_0\circ \tau_{\CL_0}^{-1}$. \\
\noindent
Consider the following diagram

\[
\xymatrix {
S_0\ar[rr]^{(\CL_{0+},\tau_{\CL_0},\delta_0,\gamma_0)}\ar[d]_i & & \wt{\CM}_{\ul\BL}^{\hat{Z}}\ar[d]^{\pi^{loc}}  \\
S\ar[rr]_{(\CL_+,g)}\ar@{.>}[urr]^{\alpha}& & \check{\CM}_{\ul\BL}^{\hat{Z}} \;.
}
\]

To prove $a)$ we have to verify that there is a map $\alpha$ that fits in the above  commutative diagram.

\noindent
We construct $\alpha : S\to\wt{\CM}_{\ul\BL}^{\hat{Z}}$ in the following way. First we take a lift $\gamma:\sigma_S^\ast \CL_+\tilde{\to} (L^+\BP)_S$ of $\gamma_0:\sigma_{S_0}^\ast \CL_{0+}\tilde{\to} (L^+\BP)_{S_0} $. To see the existence of such lift one can proceed as in  \cite[Proposition~2.2.c)]{H-V}. Namely, regarding the smoothness of $\BP$, one first observes that if a torsor gets mapped to the trivial torsor under $\check{H}^1(S_{\text{\'et}},L^+\BP)\to\check{H}^1(S_{0, \text{\'et}},L^+\BP)$, it must initially be a trivial one. Consider an $L^+\BP$-torsor $\CL_+$ over $S$. It can be represented by trivializing cover $S'\to S$ and an element $h''\in L^+\BP(S'')$, where $S''=S'\times_S S'$. A given trivialization $\gamma_0$ of $\CL_+$ over $S_0$ is given by $g_0' \in L^+\BP (S_0')$ with $p_2^\ast(g_0')p_1^\ast(g_0')^{-1}=h_0''$, where $h_0''$ is the image of $h''$ under $L^+\BP(S'')\to L^+\BP(S_0'')$ and $p_i: S''\to S'$ denotes the projection to the $i$'th factor, $i=1,2$. Take a trivialization $\beta$ of $\CL_+$, given by $f'\in L^+\BP(S')$ with $p_2^\ast(f')p_1^\ast(f')^{-1}=h''$. We modify it in the following way. Let $f_0'$ be the restriction of $f'$ to $S_0$. We have $p_2^\ast(f_0'^{-1}g_0')=p_1^\ast(f_0'^{-1}g_0')$ and therefore $f_0'^{-1}g_0'$ induces $t_0 \in L^+\BP(S_0)$. By smoothness of $\BP$ this section lifts to $t\in L^+\BP(S)$. The element $g':= t \cdot f'$ lifts $g_0'$ and satisfies $p_2^\ast(g')p_1^\ast(g')^{-1}=h''$, thus induces the desired trivialization $\gamma$. This ensures the existence of the lift $\gamma$. \\
The morphism $\alpha$ is given by the following tuple
$$
(\CL_+,\tau_\CL,\delta,\gamma):=(\CL_+,g^{-1}\circ \gamma,\tau_{\BL}\circ j^\ast \delta_0\circ\gamma^{-1}\circ g, \gamma).
$$

\noindent
Notice that

$$
\begin{CD}
i^\ast(\CL_+,\tau_\CL,\delta,\gamma)= i^\ast (\CL_+,g^{-1}\circ \gamma,\tau_{\BL}\circ j^\ast \delta_0\circ\gamma^{-1}\circ g, \gamma)\\
=(\CL_{0+},i^\ast g^{-1}\circ\gamma_0, \tau_{\BL}\circ\sigma^\ast\delta_0\circ\tau_{\CL_0}^{-1},\gamma_0)\\
=(\CL_{0+},\tau_{\CL_0},\delta_0,\gamma_0)
\end{CD}
$$
\noindent
and that $\pi^{loc}(\CL_+,\tau_\CL,\delta,\gamma)=(\CL,g)$.

\noindent
Now we prove part b). We take an \'etale covering $\CM'\to \check{\CM}_{\ul\BL}^{\hat{Z}}$ such that the universal $L^+\BP$-torsor  $\CL_+^{univ}$ admits a trivialization $\gamma': \CL_{+,\CM'}^{univ}\tilde{\to}(L^+\BP)_{\CM'}$. This yields the section

\[
\xymatrix {
& & \wt{\CM}_{\ul\BL}^{\hat{Z}} \ar[dl]_{\pi}\ar[dr]^{\pi^{loc}} &  \\
\CM'\ar[r]\ar@/^2pc/[urr]^{s'}&\check{\CM}_{\ul\BL}^{\hat{Z}}& &\hat{Z}  \;.
}
\]
\noindent
corresponding to the tuple $(\CL_+^{univ},\delta,\sigma^\ast \gamma')$.
\noindent
Consider the following diagram

\[
\xymatrix {
S_0\ar[rr]^{(\CL_{0+},\tau_{\CL_0},\delta_0,\gamma_0')}\ar[d]_i & & \CM'\ar[d]^{\pi^{loc}\circ s'}  \\
S\ar[rr]_{(\CL_+,g)}\ar@{.>}[urr]^{\alpha'}& &\hat{Z} \;.
}
\]
\noindent
We want to find $(\CL_+,\tau_{\CL},\delta,\gamma')$ with $g=\sigma^\ast \gamma'\tau_\CL^{-1}.$
First we construct
$$
(\CL_+,\tau_\CL,\delta,\gamma)\in \wt{\CM}_{\ul\BL}^{\hat{Z}}(S).
$$
Since $\sigma^\ast\gamma'=j^\ast i^\ast\gamma'=j^\ast\gamma_0'$, we take $\gamma:=j^\ast\gamma_0'$. This gives the morphism $\delta$ according to the following commutative diagram

\[
\xymatrix {
\BL & & \CL\ar[ll] \ar[rd]^{g}  \\
\sigma^\ast\BL \ar[u]_{\tau_\BL}& &\sigma^\ast\CL\ar[u]\ar[ll]^{j^\ast \delta_0} \ar[r]_{j^\ast \gamma_0'}^{\sim} & (LP)_S \;.
}
\]

\noindent
and furthermore determines $\tau_\CL$, we set 
$$
y:=\left(\CL_+,g^{-1}j^\ast\gamma_0',\tau_{\BL}\circ j^\ast \delta_0 \circ j^\ast\gamma_0'^{-1}\circ g,j^\ast \gamma_0'\right)\in 
\wt{\CM}_{\ul\BL}^{\hat{Z}}(S)
$$
with $\pi^{loc}(y)=(\CL_+, g)\in \hat{Z}(S)$. The section $s'$ sends $(\CL_{0+},\tau_{\CL_0},\delta_0,\gamma_0')$ to 
$$
(\CL_{0+},\tau_{\CL_0},\delta_0,\gamma_0=i^\ast j^\ast \gamma_0'=\sigma_{S_0}^\ast \gamma_0')=i^\ast y\in\wt{\CM}_{\ul\BL}^{\hat{Z}}(S_0).
$$
\noindent
Consider the point $$\pi(y)=(\CL_+,\tau_\CL,\delta)\in \check{\CM}_{\ul\BL}^{\hat{Z}}(S)$$ with $i^\ast \pi(y)= (\CL_{0+},\tau_{\CL_0},\delta_0)$. Then, since $\CM'\to \check{\CM}_{\ul\BL}^{\hat{Z}}$ is \'etale, there is a unique $\gamma':\CL\tilde{\to}(L^+\BP)_S$ with $i^\ast\gamma'=\gamma_0'$. Note that $$\gamma:=\sigma^\ast \gamma'=j^\ast i^\ast \gamma'=j^\ast \gamma'.$$
\noindent
This ensures the existence of $\alpha'$ which is given by $(\CL_+,\tau_\CL,\delta,\gamma')$.
To see the uniqueness let $(\CL_+,\tau_\CL,\delta,\gamma')\in\CM'(S)$ with $i^\ast(\CL_+,\tau_\CL,\delta,\gamma')=(\CL_0,\tau_{\CL_{0+}},\delta_0,\gamma'_0)$ and

$$
\pi^{loc}(\CL_+,\tau_\CL,\delta,\gamma'):=(\CL_+,\sigma^\ast \gamma' \tau_{\CL}^{-1})=(\CL,g).
$$

\noindent
Therefore $\CL_+$, $\tau_{\CL}=g^{-1} \sigma^\ast \gamma' = g^{-1} j^\ast \gamma_0'$ and $\delta=\tau_\BL \circ j^\ast\delta_0\circ j^\ast\gamma_0'^{-1}\circ g$ are uniquely determined and provide a point in $\check{\CM}_{\ul\BL}^{\hat{Z}}(S)$ and then $\gamma'$ is also uniquely determined.
\end{proof}
\section{ُSome applications of the local model}\label{SectApplocations}

Bellow we discuss some applications of the local model theorem \ref{ThmRapoportZinkLocalModel}.

\subsection{Local properties of R-Z spaces}\label{SubSectLocal properties of R-Z spaces}

Consider the Serre conditions $S_i$ and $R_i$ in the sense of \cite[\S 5.7 and 5.8]{EGAIV}. 
We say that a group $\BP$ satisfies $S_i$ (resp. $R_i$) if all singularities occurring in the orbit closures of the orbits under the $L^+\BP$ action on $\CF\ell_\BP$ satisfy  $S_i$ (resp. $R_i$). For example a parahoric $\BP$ with tame $P$ satisfies $S_i$ for all $i$, as well as $R_0$ and $R_1$, according to \cite[Theorem~8.4]{PR2}.

\begin{corollary}\label{PropLocModel}
We have the following statements:\\
\begin{enumerate}

\item[a)]
The Rapoport-Zink space $\check{\CM}_{\ul\BL}^{\hat{Z}}$ is flat over its reflex ring $R_{\hat{Z}}$ if $\zeta$ is not a zero divisor in $\CO_{\hat{Z}}$.
\item[b)]
The Rapoport-Zink space $\check{\CM}_{\ul\BL}^{\hat{Z}}$ satisfies $S_i$ (resp. $R_i$) if $\hat{Z}$ satisfies $S_i$ (resp. $R_i$). In particular $\check{\CM}_{\ul\BL}^{\hat{Z}}$ satisfies $S_i$ if $\BP$ satisfies $S_i$ and $\zeta$ is not a zero divisor in $\CO_{\hat{Z}}$.

\end{enumerate}
\end{corollary}

\begin{proof}

The part a) is clear according to Theorem \ref{ThmRZisLFF} and Theorem \ref{ThmRapoportZinkLocalModel} and that the henselization morphism $R\to R^h$ is faithfully flat. \\
\noindent
Again the first statement of part b) follows from Theorem \ref{ThmRapoportZinkLocalModel} and the fact that satisfying $R_i$ (resp. $S_i$) is an \'etale local property. For the second statement, since the element $\zeta$ is regular and by definition $Z$ satisfies $S_i$, we argue that $\hat{Z}$ satisfies $S_i$.      

\end{proof}

\noindent
\subsection{Kottwitz-Rapoport stratification}\label{SubsectKottwitz-Rapoport stratification} In the rest of this section we assume that $\BP$ is a parahoric group scheme. This in particular implies that $\CF\ell_\BP$ is ind-projective. Fix a boundedness condition $\hat{Z}$ and a local $\BP$-shtuka $\ul\BL$ as before and set $\check{\CM}:=\check{\CM}_{\ul\BL}^{\hat{Z}}$.
\noindent
The local model diagram induces a smooth morphism

$$
\check{\CM}\to [L^+\BP\backslash \hat{Z}]
$$
\noindent
of formal algebraic stacks; see \cite[Section~2.1]{EsmailDissertation}. This further induces a morphism $\check{\CM}_s\to L^+\BP\backslash Z$ of the special fibers. As a set $L^+\BP\backslash Z$, is given by $\{x_\omega\}_\omega$ corresponding to the orbits of the $L^+\BP$-action on $Z$, indexed by a finite subset of the affine Weyl group $\wt W$ associated with $P$. The preimages of $x_\omega$ define a stratification $\{\check{\CM}^\omega\}_\omega$ on $\check{\CM}$ by smooth sub-schemes $\check{\CM}^\omega$. The closure relation between these strata is given by the natural Bruhat order on $\wt W$ and one may set $\check{\CM}^{\preceq\omega}:=\cup_{\lambda\preceq\omega}\check{\CM}^\lambda$.   
 
\bigskip
\noindent
\subsection{Semi-simple trace of Frobenius}\label{Subsect Semi-simple trace and Frob} Let $S:=\Spf R$ be a formal spectrum of a complete discrete valuation ring $R$, with special point $s$ and generic point $\eta$, with residue fields $k:=\kappa(s)$ and $K:=\kappa(\eta)$, respectively. Consider the Galois group $\Gamma = \Gal(\kappa(\ol\eta)/\kappa(\eta))$ and the inertia subgroup $I:=\ker\left(\Gamma \to \Gal(\kappa(\ol s)/\kappa(s))\right)$, where $\kappa(\ol s)$ is the residue field of the normalization $\ol S$ of $S$ in $\kappa
(\eta)$.


For a formal scheme $\FX$, locally of finite presentation over $R$, we let $\FX_\eta$ denote the associate $K$-analytic space in the sense of \cite{Berk}. Let $\FX_\ol s$ denote the geometric special fiber. To the following diagram



$$
\CD
\FX_\ol\eta @>>> \FX @<<< \FX_\ol s\\
@VVV @VVV @VVV\\
\ol\eta @>>>\ol S @<<< \ol s
\endCD
$$
\noindent
one associate the following functor of nearby cycles 
$$
\CD
R\Psi^\FX: D_b^c(\FX_\ol\eta ,\ol\BQ_\ell)\to D_b^c(\FX_\ol s \times \eta, \ol\BQ_\ell)\\ ~~~~~~~~~~~~~\CF \mapsto\ol i^\ast R\ol j_\ast\CF_\ol\eta,
\endCD
$$
see \cite{Berk}. Here $\CF_\ol\eta$ denotes the restriction of $\CF$ to $\FX_\ol\eta$.

\bigskip
\noindent
Consider the spectral sequence of vanishing cycles 

\begin{equation}\label{Eq_VanishingCyclesSpectralSeq}
E_2^{p,q}:=\Koh^p(\FX_\ol s, R^q\Psi^\FX\ol\BQ_\ell)\Rightarrow H^{p+q}(\FX_\ol\eta,\ol\BQ_\ell).
\end{equation}
This spectral sequence is equivariant under the action of the Galois group $\Gamma$. Note that the induced filtration on $\CV=\Koh^\ast(X_\ol\eta,\BQ_\ell)$ is admissible in the following sense that it is stable under the Weil group action and that $I$ operates on $gr_{\bullet}^\CW(\CV)$ through a finite quotient. \forget{Recall that for an arbitrary $\BQ_\ell$-representation of $W$ with admissible increasing filtration $\CW$.  The semisimple trace of the j-Frobenius on $\CV$ is defined as follows

$$
tr^{ss}(\sigma^j|\CV):=\sum_k \Tr(\sigma^j|(gr_k^\CW(\CV))^I) 
$$ 

\noindent
} This allows to define the semi-simple trace of Frobenius $tr^{ss}(Fr_q; R\Psi_x^\FX\ol\BQ_\ell)$ on the stalk $R\Psi_x^\FX$ of the sheaf of nearby cycles at $x$.\\ 
Let us set $\check{\CM}:=\check{\CM}_{\ul\BL}^{\hat{Z}}$ and $\kappa=\kappa_{\wh{Z}}$. Let $\kappa_r/\kappa$ be a finite extension of degree $r$.  One may introduce the following function
$$
\CD
\check{\CM}(\kappa_r)\to \ol\BQ_\ell\\
~~~~~~~~~~~~~~~~~~~~~~~~~~~x\mapsto tr^{ss}(Frob_r; R\Psi_x^{\check{\CM}}\ol\BQ_\ell).
\endCD
$$

Here $Frob_r$ denotes the geometric Frobenius in $\Gal(\ol\kappa_r/\kappa_r)$. Let $y$ be the image of a point $y'$ in $\CM'$ above $x$, under $\pi^{loc} \circ s'$. 
Since semi-simple trace  is \'etale local invariant, we have 
$$
tr^{ss}(Frob_x; R\Psi_\ol x^{\check{\CM}}(\ol\BQ_\ell))=tr^{ss}(Frob_r; R\Psi_y^{\hat{Z}}\ol\BQ_\ell).
$$
 Note that when $\wh Z$ comes from a cocharacter $\mu$ of $P$, the right hand side admits a description in terms of the associated Bernstein
function $z_{\mu,r}$, lying in the center of the corresponding (parahoric) Hecke algebra, according to the Kottwitz's conjecture; see \cite{HR1}. 

\bigskip

\noindent
\subsection{Lang's cycles on $\check{\CM}_\eta$}\label{SubsectLang's cycles} The assumption that $\BP$ is parahoric, implies that $\CF\ell_\BP$ is ind-projective. Consequently, the boundedness condition is provided by a proper sub-scheme $\hat{Z}\subseteq \wh{\CF\ell}_\BP$. This allows to push forward cycles along $\check{\CM}\times_{R_{\hat{Z}}}\hat{Z}\to \check{\CM}$.\\
Assume that the local model, and consequently $\check{\CM}$ are flat over $R_{\hat{Z}}$, see Corollary \ref{PropLocModel}. Let $\CZ$ be the scheme theoretic image of $\CM'$ under the map $f$ given by the following diagram
\begin{eqnarray}\label{Eq_Etale_Roof}
\xymatrix {
\CM'\ar@/^2pc/[drrr]^{\text{\'et}}\ar@/^-2pc/[dddr]_{\text{\'et}} \ar[dr]^f & & & \\
 & \check{\CM}\times_{R_{\hat{Z}}}\hat{Z}\ar[rr]\ar[dd]^{~~~~~~~~~~~\square} & & \hat{Z}\ar[dd]  \\
& & & \\
& \check{\CM} \ar[rr]& & R_{\hat{Z}} \;.
}
\end{eqnarray}
\noindent
See proof of Theorem \ref{ThmRapoportZinkLocalModel}. Recall that the morphism $\CM'\to \hat{\CM}$ is a trivializing cover for the universal $L^+\BP$-torsor $\CL_+^{univ}$ over $\check{\CM}$, and hence surjective. Consider the spaces $\check{\CM}_\ol\eta$ and $\hat{Z}_\ol\eta$, and view $\CZ$ as a correspondence $\hat{Z} \rightsquigarrow \check{\CM}$, thus regarding the above spectral sequence \ref{Eq_VanishingCyclesSpectralSeq}, the pull-back and push-forward of the cycles on the special fibers, in the usual sense, induce the following morphism

\forget
{
$$
\CZ: H_c^i((\check{\CM}_{\ul\BL}^{\hat{Z}})^{ad},\ol\BQ_\ell)\to H^i((\hat{Z})^{ad},\ol\BQ_\ell)
$$

and
$$
\CZ^{tr}: H_c((\hat{Z})^{ad},\ol\BQ_\ell)\to H_c^i((\check{\CM}_{\ul\BL}^{\hat{Z}})^{ad},\ol\BQ_\ell).
$$

$$
\CZ: H_c^i(\check{\CM}_\ol\eta,\ol\BQ_\ell)\to H^i(\hat{Z}_\ol\eta,\ol\BQ_\ell)
$$

and
$$
\CZ^{tr}: H^i(\hat{Z}_\ol\eta,\ol\BQ_\ell)\to H_c^i(\check{\CM}_\ol\eta,\ol\BQ_\ell).
$$

}

\begin{equation}
\CZ: H^i(\hat{Z}_\ol\eta,\ol\BQ_\ell)\to H_c^i(\check{\CM}_\ol\eta,\ol\BQ_\ell).
\end{equation}

Note that the special fiber $Z$ of $\hat{Z}$ is a union of affine Schubert varieties. Let us assume that $Z$ is irreducible. In this case we have $Z=\CS(\omega)$ for some $\omega\in \tilde W$. Composing with the $\Gal(\ol\eta/\eta)$-equivariant isomorphism $H^i(Z, R\Psi\ol\BQ_\ell)\cong H^i(\hat{Z}_\eta, \ol\BQ_\ell)$, yields

\begin{equation}\label{Eq_SchubertNearbyCycles}
H^i(Z, R\Psi\ol\BQ_\ell) \to H_c^i(\check{\CM}_\ol\eta,\ol\BQ_\ell).
\end{equation}

To study the cohomology of $Z$ one can proceed by using decomposition theorem \cite{BBD} in the following context. First, consider Demazure resolution $\varrho: \Sigma(\omega)\to Z$ of singularities of $Z$, see \cite[Corollary~3.5]{Richarz}. The Demazure variety $\Sigma(\omega)$ can be viewed as a tower of iterated fiberations with homogeneous projecrive fibers; see \cite[Remark~2.9]{Richarz}. Thus one can implement Leray spectral sequence, together with the decomposition theorem, to analyze the cohomology of $Z$. Assuming that the resolution $\varrho$ is semi-small (this is for example the case when $\BP$ is constant, i.e. $\BP:=G\times_{\BF_q}\Spec \BF_q\dbl z\dbr$ for split reductive $G$ over $\BF_q$, or when the Schubert variety $Z=\CS(\omega)$ is (quasi-)minuscule), we can even argue that the motive $M(Z)$ is a direct summand of $M(\Sigma(\omega))$, and consequently mixed Tate. Here $M(X)$ denote the motive associated with a scheme $X$, in the sense of \cite{FSV}, in the motivic category $DM_{gm}(k,\BQ)$. Also see \cite{CM} and \cite{A_Ha} for the motivic versions of the decomposition theorem and Leray-Hirsch theorem, respectively. In particular we see that $Ch^i(Z)$ is finite. Let us further assume that $Z$ is smooth; see \cite{HR2} for a list of such Schubert varieties. Then the isomorphism $Ch^i(Z)\cong\Hom (M(Z),\BQ(i)[2i])$, see \cite[Theorem 14.16 and Theorem 19.1]{MVW}, gives 

$$
M(Z) \to Ch^i(Z)\dual \otimes \BQ(i)[2i].
$$

\bigskip
\noindent
Here $Ch^i(Z)\dual$ denotes the dual vector space. Summing up over all $i$, we obtain the following decomposition

$$
M(Z) \tilde{\to} \oplus_{i} Ch^i(Z) \otimes \BQ(i)[2i],
$$
\noindent
Finally notice that composing \ref{Eq_SchubertNearbyCycles} with the cycle class map, one obtains
$$
Ch^\ast(Z)\to H_c^\ast(\hat{\CM}_\ol\eta, \ol\BQ_\ell).
$$
This endows $H_c^\ast(\hat{\CM}_\ol\eta, \ol\BQ_\ell)$ with a module structure over finitely generated algebra $Ch^\ast(Z)$. In particular we may formulate the following corollary.

\begin{corollary}
Let $\hat{Z}$ be a boundedness condition with irreducible smooth special fiber $Z$. Let $\check{\CM}:=\check{\CM}_{\BL}^{\hat{Z}}$  denote the corresponding Rapoport-Zink space for local $\BP$-shtukas. Then, regarding the above construction $H_c^\ast(\check{\CM}_\eta)$ carries a module structure over finite $\BQ$-algebra $Ch^\ast(Z)$.
\end{corollary} 

\forget{
\comment{justify if is canonical!}
}
\begin{remark}

To deduce a similar statement when the special fiber $Z$ is not irreducible, one may use the Rapoport-Zink weight spectral sequence. Alternatively, one can implement the filtration described in \cite[Lemma~3.3]{A_Ha}. This in particular shows that $M(Z)$ is mixed Tate and consequently the corresponding Chow groups are finite.  

\end{remark}

\forget{

\begin{theorem}
Assume that $\hat{Z}$ comes by base change from a Schubert variety $S(\mu)\subseteq \CF\ell_\BP$ for a minuscule coweight $\mu$. Then the motive $M(\check{\CM}_\alpha)$ of an irreducible component $\check{\CM}_\alpha$ of $\check{\CM}_\eta$ is pure Tate.
\end{theorem}

\begin{proof}
In  this case $\hat{Z}$ and thus $\check{\CM}$ are smooth. Consider the motive $M(\hat{Z}_\eta)$ associated with $\hat{Z}_\eta$; see \cite{AIS}. Let $\check{\CM}_\alpha$ be an irreducible component of $\check{\CM}_\eta$. According to \ref{Eq_Etale_Roof} we observe that $M(\check{\CM}_\alpha)$ is a summand of $M(\hat{Z}_\eta)$ and therefore is pure Tate; see \cite{A_Ha}[Prop. 02].
\end{proof}

\noindent
\textbf{Satake Theory}

\[
\xymatrix {
& & \wt{\CM}_{\ul\BL}^{\hat{Z}} \ar[ddll]_{\pi}\ar[ddrr]^{\pi^{loc}}\ar[d]_\cap & &\\
& & \wt{\CF\ell}_\BP \ar[dl]_{\pi}\ar[dr]^{\pi^{loc}} & & \\
\check{\CM}_{\ul\BL}^{\hat{Z}}\ar[r]^{\subset}&\wh{\CF\ell}_\BP& &\wh{\CF\ell}_\BP & \hat{Z} \ar[l]_\supset\;.
}
\]

}

\bigskip

%
%

\Verkuerzung
{
\vfill

\begin{minipage}[t]{0.35\linewidth}
\noindent
Esmail Arasteh Rad\\
Universit\"at M\"unster\\
Mathematisches Institut \\
erad@uni-muenster.de
\\[1mm]
\end{minipage}
\forget{
\begin{minipage}[t]{0.35\linewidth}
\noindent
Utsav Choudhury\\
Department of Mathematics\\ and Statistics\\ IISER, Kolkata.
\\
utsav.choudhury@iiserkol.ac.in\\
\\[1mm]
\end{minipage}
\begin{minipage}[t]{0.35\linewidth}
\noindent
Somayeh Habibi\\
School of Mathematics,\\ 
Institute for Research\\ 
in Fundamental Sciences
(IPM)\\
shabibi@ipm.ir
\\[1mm]
\end{minipage}
}
}
{}

\end{document}